\documentclass[12pt]{amsart}

\usepackage{rotating}
\usepackage{pdfpages}
\usepackage{mathdots}
\usepackage{comment}
\usepackage{graphicx,eepic}
\usepackage{amsmath,amsfonts,amssymb, amsthm}
\usepackage{float}
\restylefloat{table}
\restylefloat{figure}

\newcommand{\BbR}{\mathbb{R}}
\newcommand{\BbN}{\mathbb{N}}

\newcommand{\la}{\lambda}

\newcommand{\be}{\beta}

\newcommand{\e}{\varepsilon}

\renewcommand{\epsilon}{\varepsilon}

\renewcommand{\hat}{\widehat}
\renewcommand{\tilde}{\widetilde}

\newcommand{\Z}{{\mathcal{Z}}}

\renewcommand{\S}{{\mathcal{S}}}
\newcommand{\OO}{{\mathcal{O}}}
\newcommand{\U}{{\mathcal{U}}}
\newcommand{\I}{{\mathcal{I}}}

\renewcommand{\emptyset}{{\varnothing}}

\newtheorem{thm}{Theorem}[section]
\newtheorem{lemma}[thm]{Lemma}

\newtheorem{prop}[thm]{Proposition}
\newtheorem{cor}[thm]{Corollary}

\newtheorem{example}[thm]{Example}
\theoremstyle{remark}
\newtheorem{rmk}[thm]{Remark}

\theoremstyle{defn}
\newtheorem*{defn}{Definition}

\begin{document}

\title[On a family of self-affine sets]{On a family of self-affine sets: topology, uniqueness, simultaneous expansions}

\dedicatory{To the memory of David Broomhead}

\author{Kevin G. Hare}
\address{Department of Pure Mathematics \\
University of Waterloo \\
Waterloo, Ontario \\
Canada N2L 3G1}
\email{kghare@uwaterloo.ca}
\thanks{Research of K. G. Hare supported, in part by NSERC of Canada.}
\thanks{Computational support provided in part by the Canadian Foundation for Innovation,
    and the Ontario Research Fund.}
\author{Nikita Sidorov}
\address{School of Mathematics \\
The University of Manchester \\
Oxford Road, Manchester\\
 M13 9PL, United Kingdom.}
 \email{sidorov@manchester.ac.uk}
\date{\today}



\begin{abstract}
Let $\be_1,\be_2>1$ and $T_i(x,y) = \bigl(\frac{x+i}{\beta_1}, \frac{y+i}{\beta_2}\bigr),\ i\in\{\pm1\}$. Let
$A := A_{\beta_1, \beta_2}$ be the unique
    compact set satisfying $A = T_{1}(A) \cup T_{-1}(A)$.
In this paper we give a detailed analysis of $A$, and the parameters $(\beta_1, \beta_2)$ where
    $A$ satisfies various topological properties. In particular, we show that if $\beta_1<\beta_2<1.202$,
    then $A$ has a non-empty interior, thus significantly improving the bound from \cite{DJK}. In the opposite direction,
    we prove that the connectedness locus for this family studied in \cite{Sol} is not simply connected.
    We prove that the set of points of $A$ which have a unique address has positive Hausdorff dimension for all $(\beta_1,\beta_2)$.
    Finally, we investigate simultaneous $(\be_1,\be_2)$-expansions of reals, which were the initial motivation
    for studying this family in \cite{Gunturk}.
\end{abstract}

\maketitle

\section{Introduction}
\label{sec:intro}

Let $T_i(x,y) = \bigl(\frac{x+i}{\beta_1}, \frac{y+i}{\beta_2}\bigr)$ for $i=\pm1$ and $A := A_{\beta_1, \beta_2}$ be the {\em attractor}
of the iterated function system (IFS) $\{T_{-1}, T_1\}$, i.e., the unique compact set satisfying $A = T_{1}(A) \cup T_{-1}(A)$. It is well
known that $A$ is either connected or totally disconnected \cite{Hata}.

Figures suggest that when $\be_1$ and $\be_2$ are ``sufficiently small'', $A_{\be_1,\be_2}$ is connected and if, in addition,
they ``very small indeed'', then $A_{\be_1,\be_2}$ has a non-empty interior -- see Figure~\ref{fig:attractors}.
The main purpose of this paper is to make such statements quantifiable, thus expanding results from \cite{DJK, Sol}.

\begin{figure}
\includegraphics[width=110pt,height=110pt]{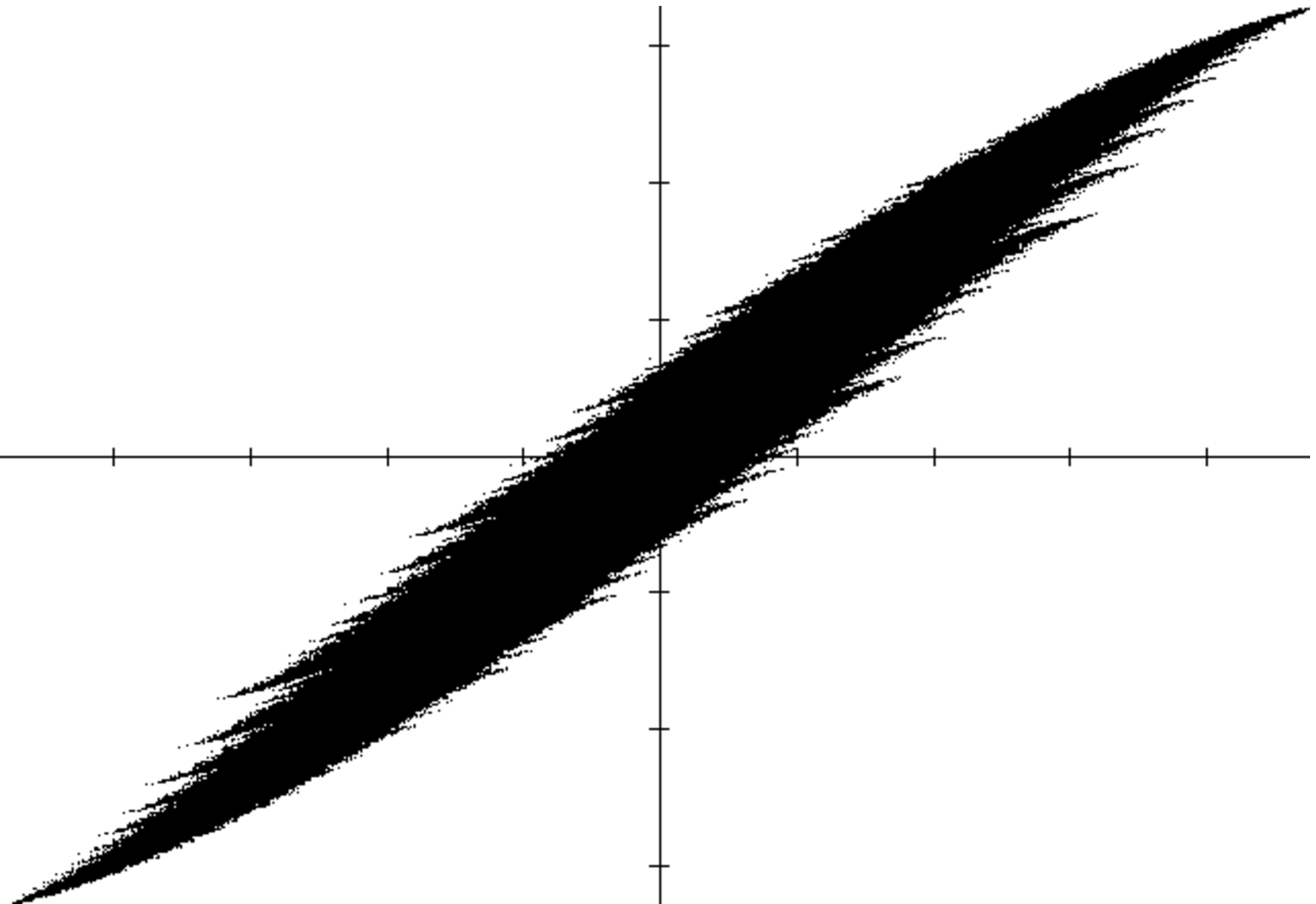}
\includegraphics[width=110pt,height=110pt]{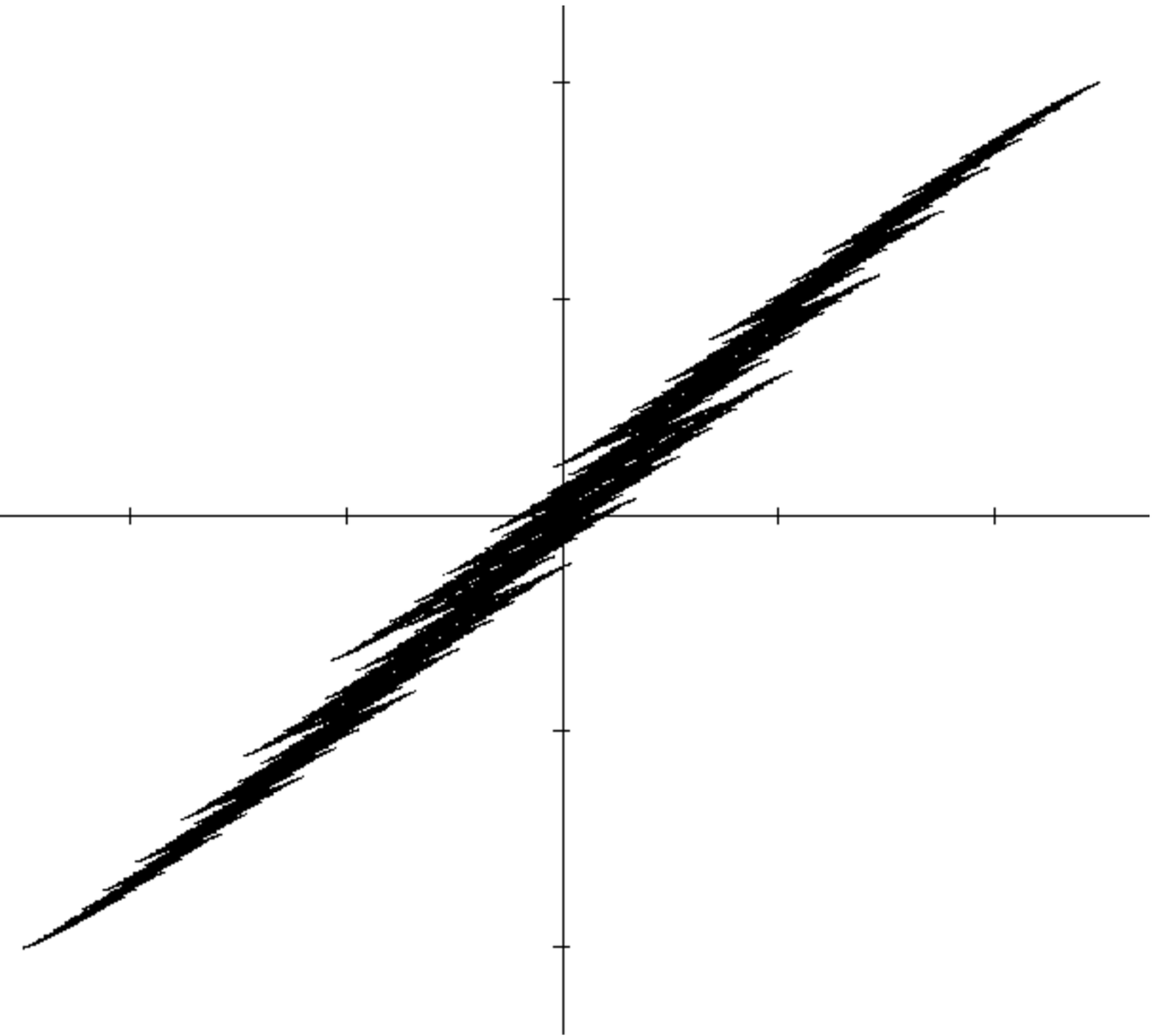}
\includegraphics[width=110pt,height=110pt]{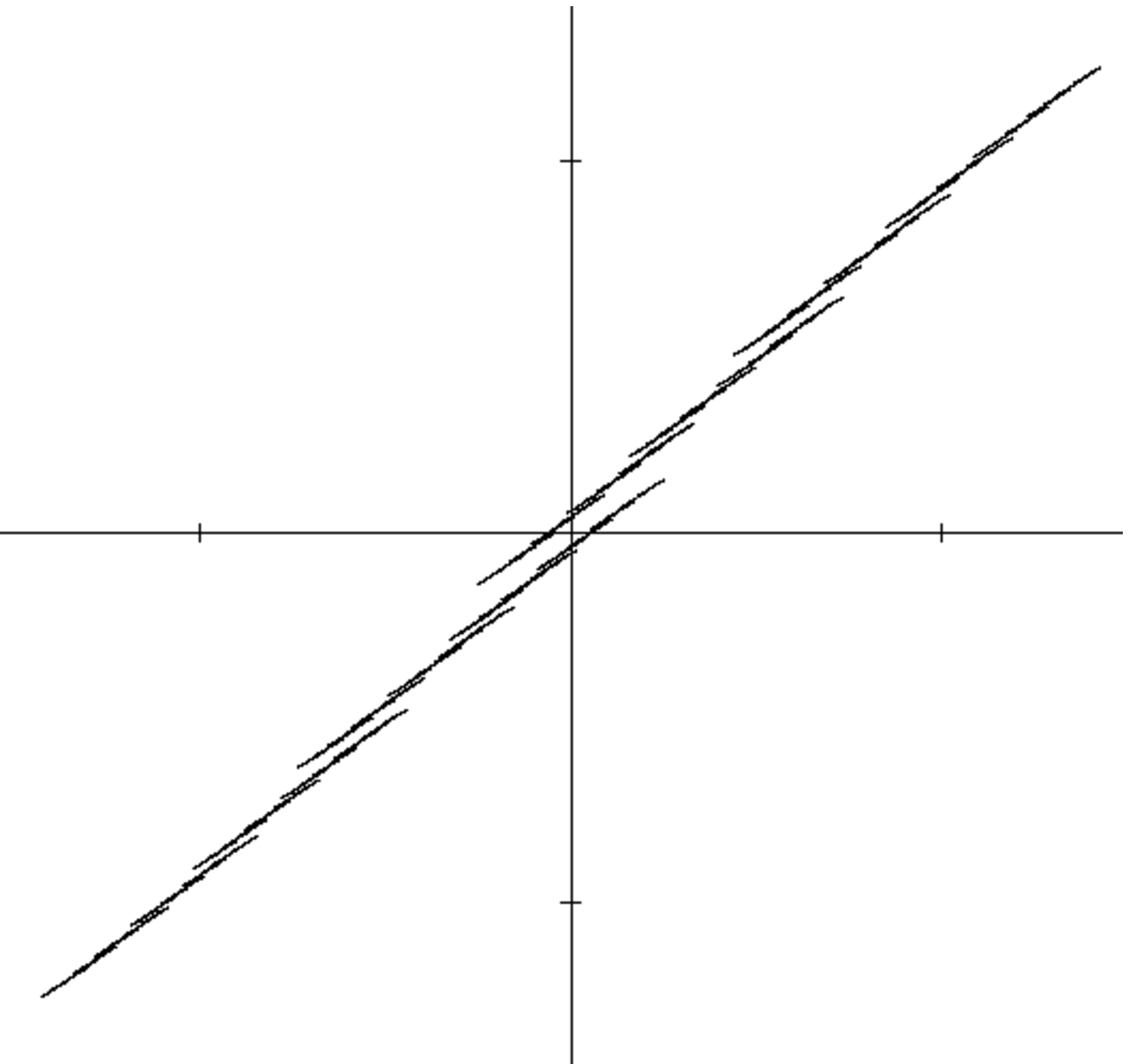}
\caption{$A_{1.2,1.3},\ A_{1.4,1.5}$ and $A_{1.7,1.8}$.}
\label{fig:attractors}
\end{figure}

Clearly, if $\beta_1 = \beta_2$ then this set is either a Cantor set if $\beta_1 = \beta_2 > 2$ or a
    one-dimensional segment otherwise. Hence, the set is trivial. So, without loss of generality we will assume that
    $\beta_1 \neq \beta_2$ throughout this paper.

For ease of notation, we will let $\la = 1/\beta_1$ and $\mu = 1/\beta_2$.
Some solutions and discussions are simplified using $\la$ and $\mu$, and some
    with $\beta_1$ and $\beta_2$.
As such, we will use these notations interchangeably.

We will denote $-1$ by $m$ (for ``minus'') and $+1$ by $p$.
A {\em word} $w \in \{p, m\}^n$ is a sequence of $p$ and $m$ of length $n$.
The set $\{p, m\}^*$ will be the set of all finite words, and $\{p, m\}^\BbN$ the set of
    all infinite words.
For $w = w_1 w_2 \dots w_n \in \{p, m\}^*$, we will denote by $T_w$ the map
    $T_{w_1} T_{w_2} \dots T_{w_n}$.
If $u, w \in \{p, m\}^*$, we will denote by $uw$ the concatenation of $u$ followed by $w$.
We will mean by $u w^\infty$ the infinite word $u w w w w \dots$.
We will use $\tilde\cdot$ for negation.
That is, $\tilde{p} = m$, $\tilde{m} = p$  and $\tilde w = \tilde{w_1} \tilde{w_2} \dots$.

We will define the map $s_\la: \{p, m\}^\BbN \to \BbR$ as
    $s_\la (w) = \sum_{i=1}^\infty w_i \la^i = \sum_{i=1}^\infty w_i / \beta_1^i$.
We will define the map $\pi: \{p, m\}^\BbN \to \BbR^2$ as
    $\pi(w) = (s_\la(w), s_\mu(w))$.
Thus, in this notation,
    \[
    A_{\beta_1, \beta_2} = \left\{\pi(w) : w \in \{p ,m\}^\BbN \right\}.
    \]

For a point $(x,y) \in A_{\beta_1, \beta_2}$ we will say it {\em has address $w \in \{p, m\}^\BbN$}
    if $\pi(w) = (x,y)$. It should be noted that a point $(x,y)$ may not have a unique address.

\subsection{The set $\Z$} \

We begin our study by considering the following set
\[
\Z = \{(\beta_1, \beta_2) : (0,0) \in A^{o}\},
\]
where $A^o$ is the interior of $A$. In a slightly different language, $\Z$ has been studied by
    Dajani, Jiang and Kempton who proved the following result:
\begin{thm}[\cite{DJK}]
If $1 < \beta_1, \beta_2 < 1.05$, then $(\beta_1, \beta_2) \in \Z$.
\end{thm}

In this paper we improve this result to show that
\begin{thm}
\label{thm:Z}
If $\beta_1 \neq \beta_2$ are such that
\[ \left|\frac{\beta_2^8-\beta_1^8}{\beta_2^7-\beta_1^7} \right| +
          \left|\frac{\beta_2^7 \beta_1^7 (\beta_2-\beta_1)}{\beta_2^7-\beta_1^7} \right| \leq 2,
\]
then $(\beta_1, \beta_2) \in \Z$.
\end{thm}

As a consequence, we have
\begin{cor}
\label{cor:Z<1.202}
If $1 < \beta_1, \beta_2 < 1.202$ then $(\beta_1, \beta_2) \in \Z$.
\end{cor}

We can also, in some cases, computationally check if $(\beta_1, \beta_2) \in \Z$ and
if $(\beta_1, \beta_2) \notin \Z$.
Many cases unfortunately remain unknown.
These are shown in Figure~\ref{fig:Z20}.
Those points provably in $\Z$ coming from Theorem~\ref{thm:Z} are shown in grey.
Those points provably not in $\Z$, as discussed in Lemma~\ref{lem:not Z}, are shown in black. Note that
all points above the curve $\be_1\be_2=2$ (shown in red) are not in $\Z$ either.
These results will be discussed in Section~\ref{sec:Z}.

\begin{figure}
\includegraphics[width=300pt,height=300pt]{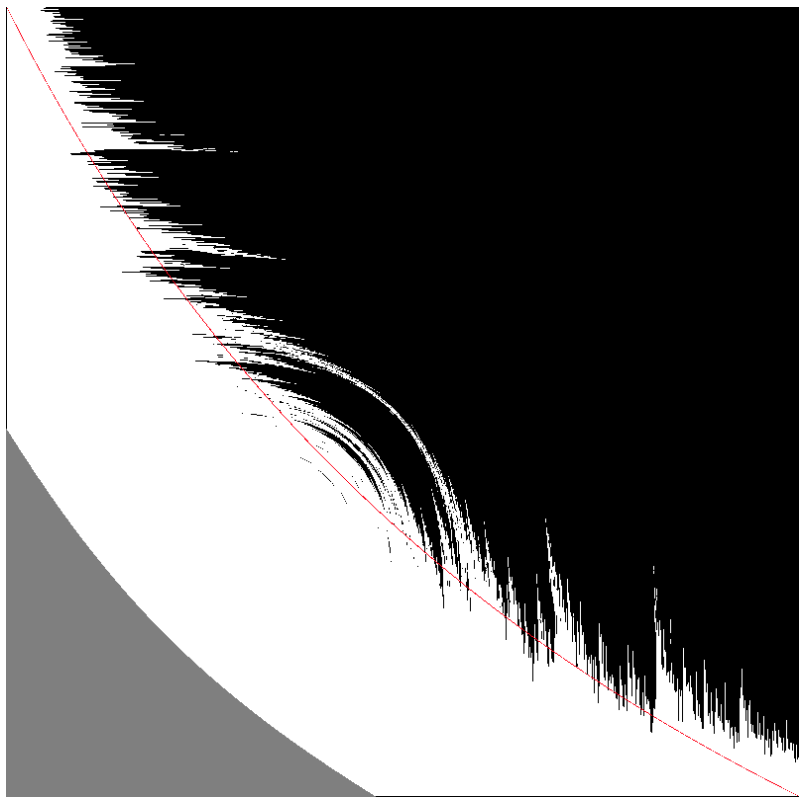}
\caption{Points known to be in $\Z$ (grey);
         points known to be not in $\Z$ (black),
         curve $\beta_1 \beta_2 = 2$ (red).
         (Level 20 approximation.)}
\label{fig:Z20}
\end{figure}

The question ``Is $(0,0) \in A^o$?'' can be easily extended to higher dimensions. Namely,
let
\[
T_i(x_1,\dots,x_m)=\left(\frac{x_1+i}{\be_1},\dots, \frac{x_m+i}{\be_m}\right),\quad i\in\{\pm1\}.
\]
Let $A_{\be_1,\dots,\be_m}$ denote the attractor of this IFS, and put
\[
\Z_m = \{(\beta_1, \dots, \beta_m) : (0,0,\dots,0) \in A_{\beta_1, \dots, \beta_m}^o\}.
\]
We show in Theorem~\ref{thm:Zm} that $\Z_m$ is always non-empty,
    first  conjectured in \cite{Gunturk}:
\begin{thm}
\label{thm:Zm}
For each $m\ge2$ there exists a $C_m > 1$ such that if
    $1 < \beta_1 < \dots < \beta_m < C_m$, then the attractor
    $A_{\beta_1, \dots, \beta_m}$ contains a neighbourhood of $(0, \dots, 0)$.
\end{thm}

\subsection{The set of uniqueness}

In the previous study, we bounded those $\beta_1, \beta_2$ such that there is a
    neighbourhood of $(0, 0)$ contained in $A$.
We observe that if $(0, 0) \in A$ by $\pi(w) = (0,0)$, then
     $\pi(\tilde w) = (0,0)$, where, as above, $\tilde{w}$ is the negation of $w$.
In particular, $(0,0)$ does not have a unique address under $\pi$.

For the next question, we examine the other end of this spectrum, namely,
for fixed $\beta_1$ and $\beta_2$, which points $(x,y) \in A$ have a unique address $(x,y) = \pi(w)$. More precisely,
we say that $(x,y)=\pi(w)$ has a {\em unique address} if for any
$w' \in \{p, m\}^\BbN$ with $w \neq w'$ we have $\pi(w')\neq (x,y)$.
We denote by $U_{\be_1,\be_2}$ the set of all unique addresses and by $\mathcal U_{\be_1,\be_2}$
the projection $\pi(U_{\be_1,\be_2})$ and call it the {\em set of uniqueness}.

For example, if $A_{\beta_1, \beta_2}$ is totally disconnected, then $U_{\beta_1, \beta_2} = \{p, m\}^\BbN$ and
    $\U_{\beta_1, \beta_2} = A_{\beta_1, \beta_2}$. On the other hand,
    if $(\beta_1, \beta_2) \in \Z$, then $U_{\beta_1, \beta_2} \subsetneq \{p, m\}^\BbN$ and
    $\U_{\beta_1, \beta_2} \subsetneq A_{\beta_1, \beta_2}$.

In the self-similar setting (without rotations) the set of uniqueness has been studied in detail --
see, e.g., \cite{GS, KdV} for the one-dimensional case and \cite{S07} for higher dimensions.
In particular, it is proved in \cite[Theorem~2.7]{S07} that if the contraction ratios
are sufficiently close to 1, then the set of uniqueness can contain only fixed points.
As we will see, this is very different in the self-affine setting.

We show in Lemma~\ref{lem:uniq-aux} that for $\be_1 \neq \be_2$, the set of
    uniqueness is non-empty. Furthermore, the set $U_{\be, \be_2}$ has positive topological entropy
    (Theorem~\ref{thm:uniq}), $\mathcal U_{\be, \be_2}$ has positive Hausdorff dimension
    (Corollary~\ref{cor:dimension}), and has no interior points (Proposition~\ref{prop:entropy-drop})
    for all $\beta_1,\beta_2$.
We also give sufficient conditions (albeit not provably necessary) for a point in $\U_{\be_1, \be_2}$ to
    be on the boundary of $A_{\be_1,\be_2}$ (Proposition~\ref{prop:uniq-boundary}).

\subsection{Simultaneous expansions}
Put
\[
\mathcal D_{\be_1,\be_2}=\left\{x\in\mathbb R: \exists (a_n)\in\{\pm1\}^\mathbb N \mid
x=\sum_{n=1}^\infty a_n\be_1^{-n}=\sum_{n=1}^\infty a_n\be_2^{-n}\right\}.
\]
In other words,
\[
\mathcal D_{\be_1,\be_2}=A_{\be_1,\be_2}\cap \{(x,y) : y=x\}
\]
(see Figure~\ref{fig:diag}).
Studying this set was the original motivation behind the IFS
under consideration - see \cite{Gunturk, DJK}.

We prove in Section~\ref{sec:simult} the following result:
\begin{thm}
\label{thm:simult}
\begin{enumerate}
\item [(i)]For any pair $(\be_1, \be_2)$ the set $\mathcal D_{\be_1,\be_2}$ is non-empty;
\item [(ii)] If $\min\{\be_1,\be_2\} < \frac{1+\sqrt5}2$, then the Hausdorff dimension of the set
$\mathcal D_{\be_1,\be_2}>0$ is positive;
\item [(iii)] If $\max\{\be_1,\be_2\}<1.202$, then there exists a $\delta > 0.664$ such that
    $[-\delta, \delta] \subset \mathcal D_{\be_1,\be_2}$.
\end{enumerate}
\end{thm}

\subsection{The set $\OO$ and $\S$} \

When studying iterated function systems, a common property that
    is investigated is if $A$ satisfies the open set condition.

\begin{defn}\label{def:OSC}
Let $A$ be the unique compact set such that $A = F_1(A) \cup \dots \cup F_k(A)$,
    where the $F_i$ are linear contractions.
We say that $A$ {\em satisfies the open set condition (OSC)} if there exists a non-empty open set $O$  such that
\begin{itemize}
\item $F_i(O) \subset O$ for all $i$;
\item $F_i(O) \cap F_j(O) = \emptyset$ for all $i \neq j$.
\end{itemize}
\end{defn}

An even stronger property is that of a set being totally disconnected.

\begin{defn}
We say that a set $A$ is {\em totally disconnected} if for all $x, y \in A$, $x \neq y$,
    there exist open sets $O_x$ and $O_y$  such that
\begin{itemize}
\item $x \in O_x$
\item $y \in O_y$
\item $O_x \cap O_y = \emptyset$.
\item $A \subset O_x \cup O_y$.
\end{itemize}
\end{defn}

A set is disconnected if there exist $x$ and $y$ with the above property.
It is clear that if a set is totally disconnected then it is disconnected.
It is known for this case that $A := A_{\beta_1, \beta_2}$ is either connected
    or totally disconnected \cite{Hata}.
Hence in this case the converse is also true.
That is, if $A$ is disconnected, then it must be totally disconnected.


Put
\begin{align*}
\mathcal O&=\{(\be_1,\be_2) : \{T_{-1},T_1\}\ \text{satisfies the OSC}\},\\
\mathcal S&=\{(\be_1,\be_2) : A_{\be_1,\be_2}\ \text{is totally disconnected}\}.
\end{align*}

It is easy to see that $\S \subset \OO$. Furthermore, if $\be_1>2$ or $\be_2>2$, then
the projection of $A$ onto the $x$- (respectively, $y$-) axis is a Cantor set, whence
$(\be_1,\be_2)\in\mathcal S$. Henceforth we will assume $\be_1<2$ and $\be_2<2$.


In Theorem~\ref{thm:S struct}, we give a precise description of a curve $S_1$
such that if $(\beta_1, \beta_2)$ are
    above this curve, then $(\beta_1, \beta_2) \in \S$.
As a corollary to this Theorem, we get
\begin{cor}
\label{cor:3.129}
If $\beta_1 + \beta_2 \geq 3.1294734398566\dots$
    then $(\beta_1, \beta_2) \in \S$.
If the inequality is strict, then $(\beta_1, \beta_2) \in \OO$.
For all $\epsilon > 0$
    there exist $\beta_1$ and $\beta_2$ with  $\beta_1+\beta_2 \geq 3.1257839569901 - \epsilon$
    where $(\beta_1, \beta_2) \notin \OO$.
\end{cor}

We can also, in some cases, computationally check if $(\beta_1, \beta_2) \in \S$ and
if $(\beta_1, \beta_2) \notin \S$.
Many cases remain unknown.
The first are shown in Figure~\ref{fig:S40}.
Those points provably in $\S$ are shown in black.
These results will be discussed in Section~\ref{sec:O S}.
In Section~\ref{sec:island} we show that $\S$ is disconnected.

\begin{figure}
\includegraphics[width=300pt,height=300pt]{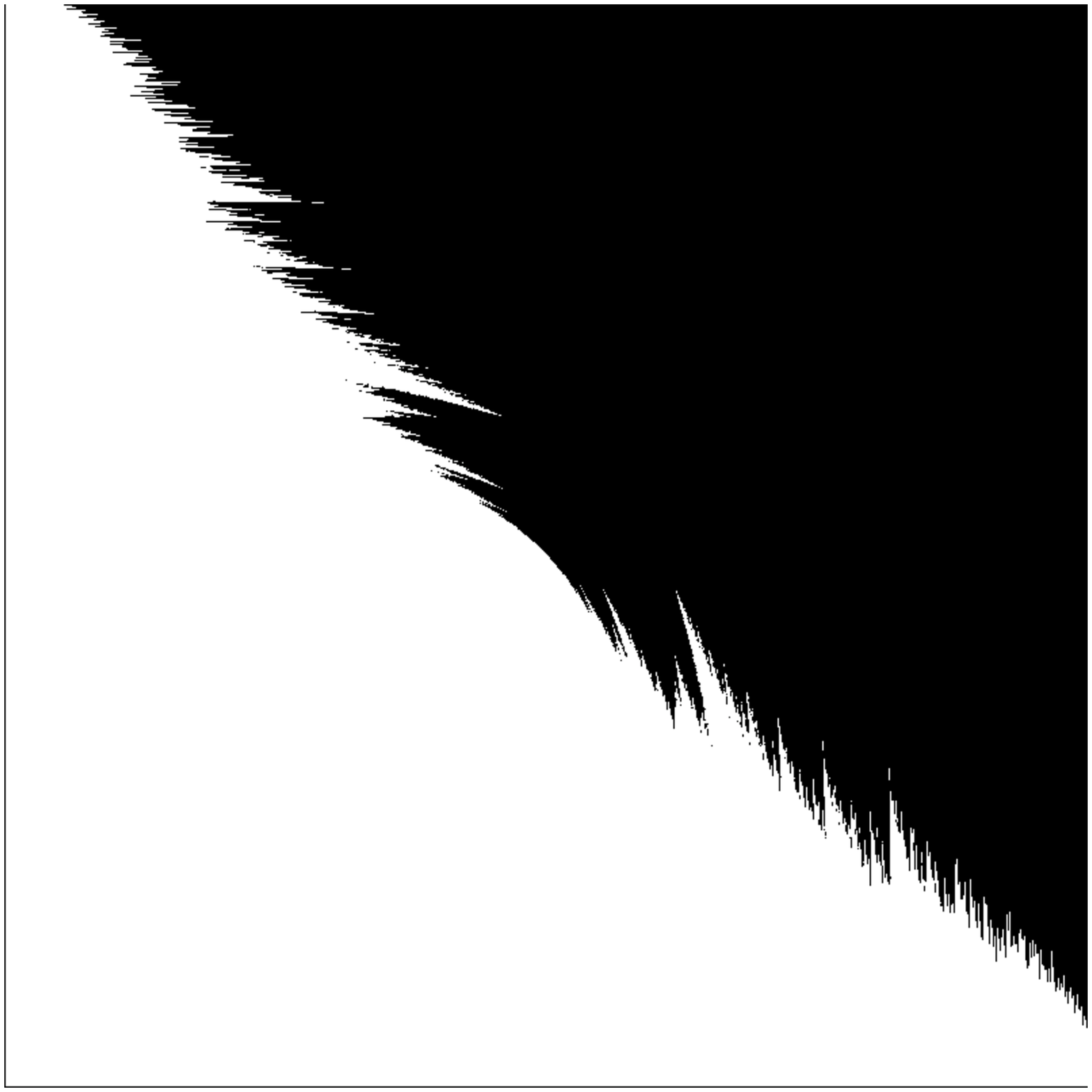}
\caption{Points known to be in $\S$ (black).
         (Level 40 approximation)}
\label{fig:S40}
\end{figure}

\subsection{Relations between sets} \

There are a number of obvious -- and some not so obvious -- relations between
    some of these sets.

Define
\[
\I = \{(\beta_1, \beta_2) : \mathrm{the\ attractor}\ A\ \text{has a non-empty interior}\}.
\]
It is clear that $\Z \subset \I$.
It is also clear that $\Z \cap \S = \emptyset$.
We know very little about $\I$, although it seems likely that $\I \cap \OO = \emptyset$.
It is not clear if $\Z \subsetneq \I$, or if in fact they are equal sets.
It is true that $\S \subsetneq \OO$, as demonstrated by the points $(\beta_1^{(n)}, \beta_2^{(n)})$
    from Theorem~\ref{thm:S struct}, which are all points in $\OO$ but not in $\S$.
All of these points $(\beta_1^{(n)}, \beta_2^{(n)})$ are points on the boundary of $\OO$, as
    shown by Solomyak \cite{Sol}.

An interesting observation to make is that there are points that are not in $\Z$ yet at the
    same time are not in $\OO$ either.

For example, let $\beta_1 \approx 1.190842710$
and $\beta_2 \approx 1.769542577$ be
roots of $x^{11}-x^{10}-x^9-x^8+x^6-x^5+x^4+x^3+x^2+x+1$.
We see by Lemma~\ref{lem:not O} that $(\beta_1, \beta_2) \notin \OO$.
As $\beta_1 \beta_2 =  2.107246878 > 2$, the Lebesgue measure of
    $A$ is $0$, hence $(\beta_1, \beta_2) \notin \Z$.

As a second example, let $\beta_1 \approx 1.122195284$
    and $\beta_2 \approx  1.776995700$ be roots of
    $x^{13}-x^{12}-x^{11}-x^9-x^8+x^7-x^6+x^5+x^4+x^3+x^2+x+1$.
Again, by Lemma~\ref{lem:not O}, $(\beta_1, \beta_2) \notin \OO$.
Since $\beta_1 \beta_2 = 1.994136194 < 2$, the Lebesgue measure argument does not work here. However, we
can, applying techniques discussed in Subsection~\ref{ssec:not Z}, show that
    $(\beta_1, \beta_2) \notin \Z$ (using a level~$25$ approximation).

This indicates that there is actually more structure here that is not fully explored.

\section{The convex hull of $A$}

Before beginning our study of properties of $A = A_{\beta_1, \beta_2}$, we will first introduce
    and study $K$, the convex hull of $A$.
The structure of $K$ will play an important role in later investigations, both
    from a computational, and a theoretical point of view.

We first give a precise description of those points that are vertices of $K$.
See for example Figure~ \ref{fig:K}.
\begin{thm}
\label{thm:K vert}
The vertices of $K$ have addresses $p^k m^\infty$ and $m^k p^\infty$ for $k = 0, 1, 2, \dots$.
\end{thm}

\begin{figure}
\includegraphics[width=400pt,height=400pt]{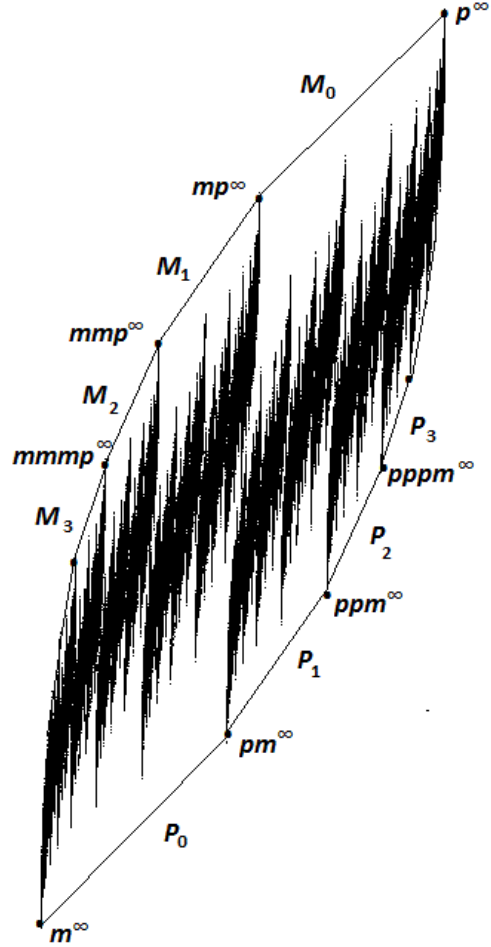}
\caption{$A_{1.85, 1.25}$ together with vertices and edges of $K$}
\label{fig:K}
\end{figure}

\begin{proof}
Without loss of generality, we may assume that $\beta_2 < \beta_1$.
It suffices to show that the line segments connecting
    $\pi(p^k m^\infty)$ and $\pi(p^{k+1} m^\infty)$ lie below $A$.
We will denote this line segment by $P_k$.
Let us begin at $k = 0$.
We must show that for any $w \in \{p, m\}^\BbN$ that the line from
$\pi(m^\infty)$ to $\pi(w)$ lies above the straight line passing
through $\pi(m^\infty)$ and $\pi(p m^\infty)$.

We notice that the line $P_0$ from $\pi(m^\infty)$ to $\pi(p m^\infty)$ is in the direction
\begin{eqnarray*}
\pi(p m^\infty) - \pi(m^\infty) & = &
\left(\frac1{\be_1} - \sum_{i \geq 2} \be_1^{-i} , \frac1{\be_2} - \sum_{i \geq 2} \be_2^{-i} \right)  \\
& & -
\left(-\frac1{\be_1} - \sum_{i \geq 2} \be_1^{-i} , -\frac1{\be_2} - \sum_{i \geq 2} \be_2^{-i} \right)  \\
                                            & = &  \left(\frac2{\be_1}, \frac2{\be_2}\right).
\end{eqnarray*}
This will have slope $s_1 = \be_1/\be_2$.

Consider now the line from $\pi(m^\infty)$ to $\pi(w)$ for $w \in \{p, m\}^\BbN$
    where $w$ not equal to $m^\infty$ and not equal to $p m^\infty$.
\begin{eqnarray*}
\pi(w) - \pi(m^\infty) & = &  \left(\sum_{i \geq 1} (a_i + 1) \be_1^{-i}, \sum_{i \geq 1} (a_i + 1) \be_2^{-i}\right).
\end{eqnarray*}
This will have slope $s_2 = (\sum_{i \geq 1} (a_i + 1) \be_2^{-i})/(\sum_{i \geq 1} (a_i + 1) \be_1^{-i})$.

It is obvious that $\pi(w)$ lies to the right of $\pi(m^\infty)$.
Hence, to show that $\pi(w)$ lies above the line $P_0$.
    it suffices to show that $s_2 > s_1$.

This will be true if and only if
\begin{equation}
\sum_{i \geq 2} (a_i + 1) \be_2^{-i+1} > \sum_{i \geq 2} (a_i + 1) \be_1^{-i+1}
\label{eq:K1}
\end{equation}
We see that the $a_i + 1$ terms are either $0$ or $2$ (and hence always non-negative).
Further $\be_2 < \be_1$ by assumption, and hence $\be_2^{-i+1} > \be_1^{-i+1}$ for all $i \geq 2$.
From this the result follows.
We know that we only get equality if $a_i + 1 = 0$ for all $a_i \geq 2$.
This cannot happen, as $w \neq m^\infty$ and $w \neq p m^\infty$.

We now proceed by induction. Consider the line $P_k$ from $\pi(p^k m^\infty)$ to $\pi(p^{k+1} m^\infty)$.
This is in the direction:
\begin{eqnarray*}
\pi(p^{k+1} m^\infty) - \pi(p^k m^\infty) & = &  (2/\be_1^{k+1}, 2/\be_2^{k+1}).
\end{eqnarray*}
This will have slope $s_1  = \be_1^{k+1}/\be_2^{k+1}$.
In particular, notice that these slopes are increasing as $k$ increases (as $\beta_1/\beta_2 >1$).

Consider a word $\pi(w)$ not equal to either $\pi(p^k m^\infty)$ or $\pi(p^{k+1} m^\infty)$.
We may assume without loss of generality that $\pi(w)$ lies to the right of $\pi(p^k m^\infty)$.
(If not, then there will exist some $k' < k$ such that $w$ lies to the right of
    $\pi(p^{k'} m^\infty)$ and to the left of $\pi(p^{k'+1} m^\infty)$.
By induction $w$ will be above this line $P_{k'}$.
As the slope are increasing, we will have that $\pi(w)$ is above the line $P_k$.)

Consider the direction from $p^k m^\infty$ to $w$.
As before, we have that
\begin{eqnarray*}
\pi(w) - \pi(p^k m^\infty) & =  &
\left(\sum_{i = 1}^k (a_i - 1) \be_1^{-i} + \sum_{i \geq k+1} (a_i + 1) \be_1^{-i},\right. \\
&& \left. \sum_{i = 1}^k (a_i - 1) \be_2^{-i} + \sum_{i \geq k+1} (a_i + 1) \be_2^{-i}\right).
\end{eqnarray*}
This will have slope
\[
s_2 = \frac{\sum_{i = 1}^k (a_i - 1) \be_2^{-i} + \sum_{i \geq k+1} (a_i + 1) \be_2^{-i}}
       {\sum_{i = 1}^k (a_i - 1) \be_1^{-i} + \sum_{i \geq k+1} (a_i + 1) \be_1^{-i}}.
\]

We have that $s_2 > s_1$ if and only if
\begin{equation}
\begin{aligned}
\sum_{i = 1}^k (a_i &- 1) \be_2^{-i+k+1} + \sum_{i \geq k+1} (a_i + 1) \be_2^{-i+k+1} \\
&>
\sum_{i = 1}^k (a_i - 1) \be_1^{-i+k+1} + \sum_{i \geq k+1} (a_i + 1) \be_1^{-i+k+1}.
\label{eq:K2}
\end{aligned}
\end{equation}
In the first sum we see that $a_i - 1$ is always $0$ or $-2$, and $\be_2^{-i+k+1} < \be_1^{-i+k+1}$.
Hence the first sum of the left hand side is always greater than or equal to that of the right hand side.
For the second sum, we see that $a_i + 1$ is always $0$ or $2$, and $\be_2^{-i+k+1} > \be_1^{-i+k+1}$.
Hence the second sum of the left hand side is always greater than or equal to that of the right hand side.
We also see that we only get equality if $w = p^k m^\infty$ or $w = p^{k+1} m^\infty$.

The points $\pi(m^k p^\infty)$ are treated in a similar way.
\end{proof}

We notice that the proof shows something stronger, namely that

\begin{cor}
\label{cor:U nonempty}
The vertices of $K$ have unique addresses.
\end{cor}

\begin{proof}
To see this, we note that equations~\eqref{eq:K1} and \eqref{eq:K2} are strict inequalities when
    $w \neq p^k m^\infty$.
\end{proof}

Recall for a finite word $w \in \{p, m\}^*$, we define $K_w = T_{w} (K)$,
    and set $K_n = \bigcup_{|w|=n} K_w$.
It is easy to see that for $w, w' \in \{p, m\}^*$ we have $K_{w w'} \subset K_w$.
In particular this shows that
    \[ A \subset \dots \subset K_{n+1} \subset K_n \subset \dots \subset K. \]
A standard result on iterated functions systems gives that
    $A = \bigcap_{n\ge1} K_n$.

We will take advantage of this construction in multiple ways throughout this paper.
For example, we will show:
\begin{itemize}
\item If $(0,0) \notin K_n$ for some $n\ge1$, then $(0,0) \notin A$.
      (Section~\ref{sec:Z}.)
\item If $T_1(K_n) \cap T_{-1}(K_n) = \emptyset$ for some $n \ge1$, then $A$ is totally disconnected.
      (Section~\ref{sec:O S}.)
\item If $T_1(K_n^o) \cap T_{-1}(K_n^o) = \emptyset$ for some $n \ge1$, then $A$ satisfies the OSC
      (Section~\ref{sec:O S}.)
\end{itemize}

\section{The set $\Z$}
\label{sec:Z}

In this section we will investigate $\Z$ in greater detail.
In Subsection \ref{ssec:Z tool} we will provide the main tool for checking
    if a point is in $\Z$ and provide a proof of Theorem~\ref{thm:Z},
    giving sufficient conditions for $(\beta_1, \beta_2) \in \Z$.
In Subsection~\ref{ssec:Zm} we will discuss the higher dimensional analogue of $\Z$.
In Subsection~\ref{ssec:not Z} we will give sufficient conditions for
    $(\beta_1, \beta_2) \notin \Z$.

\subsection{Finding points in $\Z$} \
\label{ssec:Z tool}

The main tool used to computationally check if a point $(\beta_1, \beta_2) \in \Z$
    and to find a generic bound for points in $\Z$ is a generalization and strengthening of
    Proposition 2.1 and Definition 2.1 from \cite{DJK}.

\begin{thm}
\label{thm:tool}
Let $P(x) = x^n + b_{n-1} x_{n-1} + \dots + b_0$ such that
\begin{enumerate}
\item $P(\beta_j) = 0$ for $j = 1, 2, \dots, m$,
\item $\sum_{j=0}^{n-1} |b_j| \leq 2$,
\item $b_1 = b_2 = \dots = b_{m-1} = 0$,
\item $b_0 \neq 0$.
\end{enumerate}
Then there exists a neighbourhood of $(0,\dots,0)$ in $A$, based on $\beta_1,\dots, \beta_m$.
\end{thm}

Using this theorem, it suffices to find a polynomial $P$ in terms of $\beta_1, \dots, \beta_m$
    such that the four conditions hold for all $1 < \beta_j < C$ for some $C$.
This is a purely computational search.

Consider the polynomial.
\[ P(x) = x^8 - \frac{\beta_2^8-\beta_1^8}{\beta_2^7-\beta_1^7} x^7 +
          \frac{\beta_2^7 \beta_1^7 (\beta_2-\beta_1)}{\beta_2^7-\beta_1^7} \]
A quick check shows that $P(\beta_1) = P(\beta_2) = 0$.
Further, for all $\beta_1, \beta_2 < 1.202$ then we have
\[ \left|\frac{\beta_2^8-\beta_1^8}{\beta_2^7-\beta_1^7} \right| +
          \left|\frac{\beta_2^7 \beta_1^7 (\beta_2-\beta_1)}{\beta_2^7-\beta_1^7} \right| \leq 2 \]
In fact, a stronger result can be shown.
By explicitly solving for when
\[ \left|\frac{\beta_2^8-\beta_1^8}{\beta_2^7-\beta_1^7} \right| +
          \left|\frac{\beta_2^7 \beta_1^7 (\beta_2-\beta_1)}{\beta_2^7-\beta_1^7} \right| \leq 2 \]
we find that all $\beta_1 \neq \beta_2$ in grey in Figure~\ref{fig:Z20} have
    the desired properties.

\begin{proof}[Proof of Theorem~\ref{thm:tool}]
Let $P$ have the required properties.

Let $u_{-n}, \dots, u_{-n+m-1}$ satisfy
\[
\left[ \begin{array}{c}x_1 \\ x_2 \\ \vdots \\ x_m \end{array} \right] =
b_0 \left[ \begin{array}{cccc}\beta_1^{-1} & \beta_1^{-2} & \dots & \beta_1^{-m} \\
                          \beta_2^{-1} & \beta_2^{-2} & \dots & \beta_2^{-m} \\
                              \vdots   &     \vdots   &        & \vdots       \\
                          \beta_m^{-1} & \beta_m^{-2} & \dots & \beta_m^{-m}
       \end{array} \right]
\left[ \begin{array}{c}u_{-n} \\ u_{-n+1} \\ \vdots \\ u_{-n+m-1} \end{array} \right].
\]
We see that this system will have a solution as all of the $\beta_i$ are distinct.
Moreover, we see that if the $x_j$ are sufficiently close to $0$, then
    the $u_j$ will also be sufficiently close to $0$.
Choose $\delta$ such that if $|x_j| < \delta$, then $|u_j| \leq 1$.

Set $u_{-n+m} = \dots = u_0 = 0$.
We will choose the $u_i$ and $a_i$ for $i = 1, 2, 3, \dots$ by induction, such that
    \[ u_i := a_i - \left(\sum_{k=0}^{n-1} b_{k} u_{i-n+k}\right) \]
and such that $u_i \in [-1,1]$ and $a_i \in \{-1, +1\}$.
We see that this is possible, as, by induction, $|u_j| \leq 1$ for all $j \leq i-1$.
Furthermore,
\begin{eqnarray*}
\left|\sum_{k=0}^{n-1} b_k u_{i-n+k}\right| & \leq & \sum_{k=0}^{n-1} |b_k u_{i-n+j}| \\
                                 & \leq & \sum_{k=0}^{n-1} |b_k| \\
                                 & \leq & 2,
\end{eqnarray*}
by our assumption on the $b_k$.
Hence there is a choice of $a_i$, either $+1$ or $-1$ such that
    $a_i - \sum_{k=0}^{n-1} b_k u_{i-n+k}  \in [-1, 1]$.

We claim that this sequence of $a_i$ has the desired properties.

Let $b_n = 1$ for ease of notation.
To see this, notice for $i = 1, 2$ that
\begin{eqnarray*}
\sum_{j\geq 1} a_j \beta_i^{-j}
    & = & \sum_{j\geq 1} \left(\left(\sum_{k=0}^{n-1} b_k u_{j-n+k}\right) + u_j \right)\beta_i^{-j} \\
    & = & \sum_{j\geq 1} \sum_{k=0}^{n} b_k u_{j-n+k} \beta_i^{-j}  \\
    & = & \sum_{k=0}^{n} \sum_{j\geq 1} b_k u_{j-n+k} \beta_i^{-j}  \\
    & = & \sum_{k=0}^{n} b_k \beta_i^k \sum_{j\geq 1} u_{j-n+k} \beta_i^{-j-k}  \\
    & = & \beta_i^{-n} \sum_{k=0}^{n} b_k \beta_i^k \sum_{j\geq 1} u_{j-n+k} \beta_i^{-j-k+n}  \\
    & = & \beta_i^{-n} \sum_{k=0}^{n} b_k \beta_i^k \sum_{\ell \geq -n+1} u_{\ell+k} \beta_i^{-\ell-k}  \\
    & = & \beta_i^{-n} \sum_{k=0}^{n} b_k \beta_i^k
          \left(\sum_{\ell = -n+1}^{-k} u_{\ell+k} \beta_i^{-\ell-k} +
                \sum_{\ell \geq 1} u_{\ell} \beta_i^{-\ell} \right)   \\
    & = & \left(\beta_i^{-n} \sum_{k=0}^{n} b_k \beta_i^k
                \sum_{\ell = -n+1}^{-k} u_{\ell+k} \beta_i^{-\ell-k}\right) +
          \left(\beta_i^{-n} P(\beta_i) \sum_{\ell \geq 1} u_{\ell} \beta_i^{-\ell} \right)   \\
    & = & \beta_i^{-n}
          \sum_{k=0}^{n} \sum_{\ell = -n+1}^{-k} b_k \beta_i^k u_{\ell+k} \beta_i^{-\ell-k}. \\
\end{eqnarray*}

Thus, by our construction, we have $b_1 = b_2 = \dots = b_{m-1} = 0$ and
    $u_{m-n} = \dots = 0$. Hence this simplifies to
\begin{eqnarray*}
\sum_{j\geq 1} a_j \beta_i^{-j}
    & = & \beta_i^{-n} \sum_{\ell = -n+1}^{0} b_0 u_{\ell} \beta_i^{-\ell} +
          \beta_i^{-n}
          \sum_{k=m}^{n} \sum_{\ell = -n+1}^{-k} b_k \beta_i^k u_{\ell+k} \beta_i^{-\ell-k} \\
    & = & \beta_i^{-n} \sum_{\ell = -n+1}^{0} b_0 u_{\ell} \beta_i^{-\ell} +
          \beta_i^{-n}
          \sum_{k=m}^{n} \sum_{\ell = -n+1}^{-k} b_k \beta_i^k\cdot 0 \cdot \beta_i^{-\ell-k} \\
    & = & b_0 (u_{-n+1} \beta_i^{-1} + u_{-n+2} \beta_i^{-2} + \dots + u_{-n+m+1} \beta_i^{-m}) \\
    & = & x_i,
\end{eqnarray*}
which gives the desired result.
\end{proof}

\subsection{Higher dimensional analogues of $\Z$}
\label{ssec:Zm}

We see from Theorem~\ref{thm:tool} that to prove Theorem~\ref{thm:Zm}, it
    suffices to find $P$ satisfying certain criteria.
In this subsection we will show that such a polynomial exists for all $m$.

\begin{lemma}
Let $P(x) = x^n + a_{n-1} x^{n-1} + \dots + a_0$ be such that
    $\sum_{i=0}^{n-1} |a_i| < 2$ and $P(\beta_i) = 0$ for $i = 1, 2, \dots, m$.
Let $S \subset \{0, 1, \dots, n-1\}$ be such that $|S| < n-m$.
Then there exists a neighbourhood of $(\beta_1, \dots, \beta_m)$ such that
    for all $(\hat \beta_1, \dots, \hat \beta_m)$ in this neighbourhood there
    exists a polynomial
    $\hat P(x) = x^{n} + \hat a_{n-1} x^{n-1} + \dots + \hat a_0$ where
\begin{itemize}
\item $a_s = \hat a_s$ for all $s \in S$,
\item $\sum_{i=0}^{n-1} |\hat a_i| < 2$.
\item $\hat P(\hat \beta_i) = 0$ for $i = 1, 2, \dots, m$.
\end{itemize}
\end{lemma}

\begin{proof}
Let $R$ be such that $P(x) = \prod(x-\beta_i) R(x)$.
For $\hat \beta_i$ close to  $\beta_i$, we see that the coefficients
    of $\tilde{P}(x) = \prod(x - \hat\beta_i) R(x) = x^n + \tilde a_{n-1} x^{n-1} + \dots + \tilde a_{0}$ are
close to those of $P$.
For all $s \in S$,
   let $T_s(x) = b^{(s)}_{n-1} x^{n-1} + \dots + b^{(s)}_0$ be a polynomial such that
  \begin{itemize}
      \item $b_{s'}^{(s)} = 0$ for $s' \in S, s' \neq s$.
      \item $b_{s}^{(s)} = 1$
      \item $T_s(\hat\beta_i) = 0$ for $i =\ 1, 2, \dots, m$.
  \end{itemize}
We see that such a polynomial exists as $n - |S| > m$.
Set
    \[ \hat P(x) = \tilde{P}(x) + \sum_{s \in S} (a_s - \tilde a_s) T_s(x). \]
It is easy to observe that $a_s = \hat a_s$ for $s \in S$, and that
    $\hat P(\hat\beta_i) = 0$ for $i = 1, 2, \dots, m$.
Further observe that for $\hat\beta_i$ close to $\beta_i$ we have that
    $\hat a_i$ are close to $a_i$.
Hence by continuity, we can choose a neighbourhood of $(\beta_1, \dots, \beta_m)$
such that the resulting $\hat a_i$ are close enough to $a_i$ so that
    $\sum|\hat a_i| < 2$.
We see that $\hat P$ has the desired properties.
\end{proof}

\begin{cor}
If there exists a $P \in \BbR[x]$ monic of degree at least $2m-1$, such that
    $a_1 = \dots = a_{m-1} = 0$, $\sum |a_i| < 2$ and $(x-1)^m | P$ then
    there is a neighbourhood around $(1,1,\dots, 1)$ that is contained in $\Z$.
\end{cor}

\begin{proof}
We use $S = \{1,2,\dots, m\}$ and the neighbourhood of $(1,1, \dots, 1)$.
If $a_0 = 0$, then we can use the polynomial $T_0$ to perturb $P$.
\end{proof}

\begin{thm}
Given $m \in \BbN$ there exists an $n \in \BbN$, and a polynomial
    $P(x) = x^{m n + 1} - x^{nm} + b_{m-1} x^{(m-1) n} + b_{m-2} x^{(m-2) n} +  \dots
				 + b_{0}$
     such that $(x-1)^m | P$ and $1 + \sum_{i=0}^{m-1} |{b_i}| < 2$.
\end{thm}

\begin{proof}
Let
\[
P(x) = x^{m n+1} - x^{m n} + b_{m-1} x^{(m-1) n} + ... + b_{1} x^n + b_0.
\]
We see that $(x-1)^m | P$ if and only if $P(1) = P'(1) = ... = P^{(m-1)}(1) = 0$.
Using the notation $n^{(k)} = n (n-1) (n-2) \dots (n-k+1)$,
    with $n^{(k)} = 0$ for $k > n$, consider
the $k$th derivative of $P$ with respect to $x$, with $k \geq 1$:
\begin{align*}
    P^{(k)}(x) &= (n m +1 )^{(k)} x^{nm + 1-k} -
		   (n m)^{(k)} x^{nm -k}  \\ &+
		   (n (m-1))^{(k)} b_{m-1} x^{n(m-1) -k}  +  \dots +
		   n^{(k)} b_{1} x^{n-k}.
\end{align*}
We require that $P^{(k)}(1) = 0$ for $k = 0, 1, \dots, m-1$.
Evaluating $P(x)$ at $x = 1$ gives
\begin{align}
1-1  &= b_{m-1} + b_{m-2} +  \dots + b_0.
\label{eq:nonlimit0}
\end{align}
For $k = 1, \dots, m-1$, by
    dividing by $(nm)^{(k)}$ and evaluating at $x = 1$ we have
\begin{equation}
\begin{aligned}
1-\frac{(n m +1 )^{(k)} }{(nm)^{(k)}}  &=
		   \frac{(n (m-1))^{(k)}}{(nm)^{(k)}}\ b_{m-1} +
		   \frac{(n (m-2))^{(k)}}{(nm)^{(k)}}\ b_{m-2} \\
               &+ \cdots +\frac{n^{(k)}}{(nm)^{(k)}}\ b_{1}.
\end{aligned}
\label{eq:nonlimit}
\end{equation}
Taking the limit as $n$ tends to infinity in (\ref{eq:nonlimit}), we obtain
\begin{equation}
0 = \left(\frac{m-1}{m}\right)^k b_{m-1} +
    \left(\frac{m-2}{m}\right)^k b_{m-2} + \cdots +
    \left(\frac{0}{m}\right)^k b_{0}
\label{eq:limit}
\end{equation}
for $k = 0, 1, \dots, m-1$.
Here we take $\left(\frac{0}{m}\right)^0 = 1$.
Clearly, solving (\ref{eq:limit}) for the $b_i$ is equivalent to solving the linear system:
\[
\left[
\begin{array}{c}0 \\ 0 \\ \vdots \\ 0 \end{array} \right]
=
\left[
\begin{array}{ccccc}
1 & 1 & \dots & 1 & 1 \\
\frac{m-1}{m} & \frac{m-2}{m} & \dots & \frac{1}{m} &  0 \\
\left(\frac{m-1}{m}\right)^2 & \left(\frac{m-2}{m}\right)^2 & \dots & \left(\frac{1}{m}\right)^2 &  0 \\
\vdots & & & & \vdots \\
\left(\frac{m-1}{m}\right)^{m-1} & \left(\frac{m-2}{m}\right)^{m-1} & \dots & \left(\frac{1}{m}\right)^{m-1} &  0
\end{array}
\right]
\left[
\begin{array}{c}b_{m-1} \\ b_{m-2} \\ \vdots \\ b_1 \\ b_{0} \end{array} \right].
\]
The lower left $(m-1)\times (m-1)$ submatrix is the Vandermonde matrix on the terms $\frac{m-1}{m}, \frac{m-2}{m}, \dots, \frac{1}{m}$,
    with non-zero determinant $\prod_{1 \leq i < j \leq m-1} \left( \frac{i-j}{m} \right)$.
Hence there exists an $N$ such that for all $n \geq N$ the system of equations given by \eqref{eq:nonlimit0} and \eqref{eq:nonlimit}
    has non-zero determinant, and hence will always have a solution, regardless of the
    left hand side.

We see that the system of equations given by \eqref{eq:limit} has a solution of $b_i = 0$ for $i = 0, 1, \dots, m-1$.
We see in this case that the sum $\sum_{i=0}^m |b_m| = 1$.
(Here we think of $b_m = -1$ coming from the coefficient of $x^{nm}$.)

This implies that there exists an $N_0 > N$ such that for all $n \geq N_0$ the solution to
    equations \eqref{eq:nonlimit0} and \eqref{eq:nonlimit} will have solutions
    $b_0 \approx b_1 \approx \dots \approx b_{m-1} \approx 0$ and
    $b_m \approx 1$, and $\sum_{i=0}^m |b_i| \approx 1$.

This gives a polynomial with the desired property and  proves Theorem~\ref{thm:Zm}.
\end{proof}

\begin{rmk} The fact that $C_m>1$ for all $m\ge2$ was conjectured in \cite{Gunturk}. In the same
paper the author has shown, using a simple volume covering argument, that $C_m\le 2^{1/m}$ for all $m$.
\end{rmk}

\subsection{Points not in $\Z$} \
\label{ssec:not Z}

To prove that $(\beta_1, \beta_2) \notin \Z$, it suffices to show that
    $(0,0) \notin A$.
This is clearly a sufficient condition, although it is not a necessary condition.
To see that it is not necessary, notice that $(\beta_1^{(n)}, \beta_2^{(n)})$ which we will
    discuss in Section~\ref{sec:O S} have the property that
    $(0, 0) \in A$ yet $A$ satisfies the open set condition.
Moreover, by approximating $A$ by $K$, we see that there are points,
    arbitrarily close to $(0,0)$ that are not in $K$, and hence not in $A$.
As such,  $(\beta_1^{(n)}, \beta_2^{(n)}) \notin \Z$.
See Figure~\ref{fig:touching}.

It is interesting to note that $(\beta_1^{(n)}, \beta_2^{(n)})$ is on the boundary of $\S$.
It is not clear if such an example that is not on the boundary of $\S$ would exist.

Recall that we
    denote $K_{w} = T_w(K)$ and $K_n = \bigcup_{|w| = n} K_w$. The following result holds.
\begin{lemma}
\label{lem:not Z}
If there exists an $n$ such that $(0,0) \notin K_n$, then
    $(0,0) \notin A$ and $(\beta_1, \beta_2) \notin \Z$.
\end{lemma}

\begin{figure}
\includegraphics[width=300pt,height=300pt]{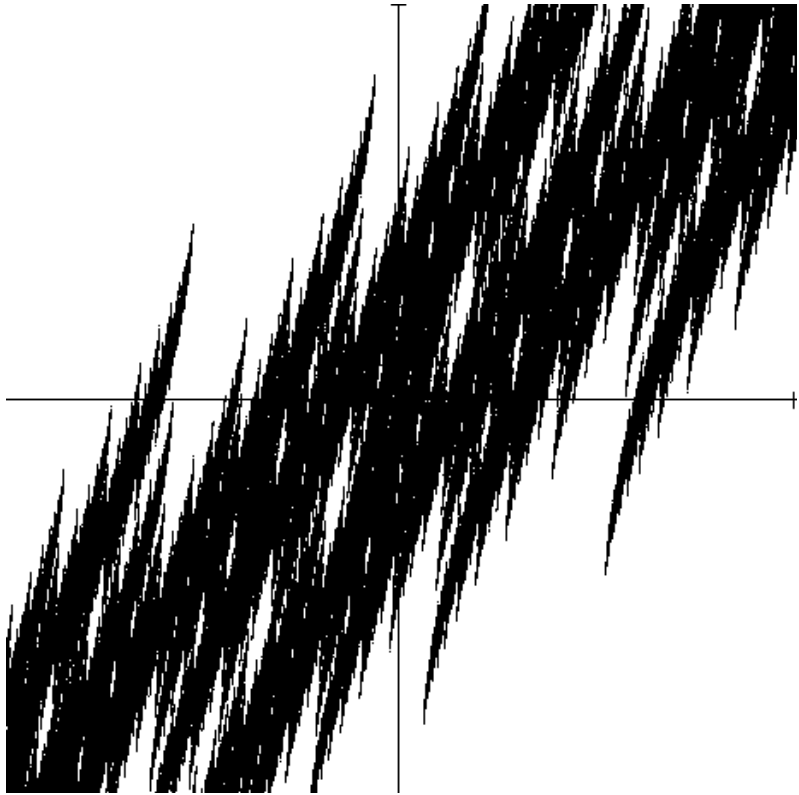}
\caption{$A_{\be_1, \be_2}$ zoomed-in around $(0,0)$, where $\be_1\approx1.57125 , \beta_2\approx1.34067$ are roots of
$x^{10}-x^9-x^8-x^7+x^6+x^5-x^4+x^3+x^2+x+1$. We have $(0,0)\in A_{\be_1,\be_2}$ but
no neighbourhood of $(0,0)$ lies in $A$.}
\label{fig:alg-not-in-z}
\end{figure}

If we were to compute the entirety of $K_n$, then it would be computationally expensive.
We observe for $w, w' \in \{p, m\}^*$ that $K_{w w'} \subset K_w$.
Hence if $(0,0) \notin K_w$ then we have that
    $(0,0) \notin K_{w w'}$ for all $w'$.
This allows for considerably more efficient computations.

In Figure~\ref{fig:Z20} we give those points that are provably not
    in $\Z$, as shown by examining $K_{20}$.
We also give those points that are provably in $\Z$ by Theorem~\ref{thm:Z}.

Note also
that if $\beta_1\beta_2>2$, then, as is well known, the Lebesgue measure of $A$ is zero,
whence all ($\beta_1,\beta_2)$ which satisfy this condition do not belong to $\Z$ either.

\begin{example}
Let $\be_1\approx1.57125 , \beta_2\approx1.34067$ be roots of $x^{10}-x^9-x^8-x^7+x^6+x^5-x^4+x^3+x^2+x+1$.
Then we have $\be_1\be_2\approx2.10653>2$, whence $(\be_1,\be_2)\notin\Z$. However,
$(0,0)$ clearly belongs to $A$ as $(0,0) = \pi((pmmmppmpppp)^\infty)$. See Figure~\ref{fig:alg-not-in-z}.
\end{example}

Observe that there is a large region of Figure~\ref{fig:Z20}, where nothing is known.

\section{The set of uniqueness}
\label{sec:uniq}

Recall that $(x,y)=\pi(w)$ has a {\em unique address} if for any
$w' \in \{p, m\}^\BbN$ with $w \neq w'$ we have $\pi(w')\neq (x,y)$.
We denote by $U_{\be_1,\be_2}$ the set of all unique addresses and by $\mathcal U_{\be_1,\be_2}$
the projection $\pi(U_{\be_1,\be_2})$ and call it the {\em set of uniqueness}.

A consequence of Corollary~\ref{cor:U nonempty} gives:
\begin{lemma}\label{lem:uniq-aux}
The set of uniqueness $\mathcal U_{\be_1,\be_2}$ is always non-empty.
\end{lemma}

Now we are ready to prove the main result of this section.
Let $E_n(\mathcal{L})$ be the number of $a_1 a_2 \dots a_n$ that are prefixes for
    some infinite word in $\mathcal{L} \subset \{p, m\}^\BbN$.
We say that $\mathcal{L}$ has positive topological entropy if $E_n(\mathcal L)$
    grows exponentially.
That is, if $\liminf_{n\to\infty}\frac{\log E_n(\mathcal L)}{n} > 0$.

\begin{thm}
\label{thm:uniq}
For any $(\be_1,\be_2)$ the set $U_{\be_1, \be_2}$ has positive topological entropy.
\end{thm}
\begin{proof}
Let $[i_1\dots i_k]$ stand for the cylinder $\{a_j\}_{j=1}^\infty \subset \{p, m\}^\mathbb N$,
    where $a_j = i_j$ for $j = 1, \dots, k$.
As $\pi(p^k m^\infty)$ has a unique address from Corollary~\ref{cor:U nonempty}, we get that
    $\text{dist}(\pi(p^km^\infty), \pi([m])>0$, where dist stands
    for the Euclidean metric. Put
\[
L_k=\min \{j\ge 1: \text{dist}(\pi([p^km^j]), \pi([m]))>0\}
\]
and $L=\max_{k\ge1} L_k$. Note that since $\pi(p^km^\infty)$ tends to $\pi(p^\infty)$ (which
is clearly at a positive distance from $\pi([m])$), the quantity $L$ is well defined.

Put
\begin{equation}\label{eq:Uprime}
\begin{aligned}
U'&=\{p^{k_0}m^{k_1}p^{k_2}\dots \mid k_0\ge 1, \ k_i\ge L,\ i\ge1\}\\
&\cup \{m^{k_0}p^{k_1}m^{k_2}\dots \mid k_0\ge 1, \ k_i\ge L,\ i\ge1\}.
\end{aligned}
\end{equation}
Clearly, $U'$ is a subshift, i.e., a closed set such that if $a_1 a_2 \dots \in U'$, then so is
    $a_{j} a_{j+1} a_{j+2} \dots\in U'$ for any $j \geq 2$.
The set $U'$ also has positive topological entropy, since
it contains the set $\prod_1^\infty \{m^Lp^{L+1}, m^{L+1}p^L\}$ which has exponential
growth. Thus, it suffices to show that any sequence in $U'$ is a unique address.

By our construction, $\pi([p^km^{k'}])$ does not intersect $\pi([m])$ provided $k'\ge L$.
This is true for all $k > 1$.
By symmetry, the same goes for $\pi([m^kp^{k'}])$ and $\pi([p])$.
This means that for $(x,y)=\pi(w_0 w_1 w_2 \dots) = \pi(p^{k_0}m^{k_1}p^{k_2}\dots)$ with
    $k_i\ge L$, we necessarily have $w_0 = p$.
Hence the problem of showing that $(x,y) = \pi(p^{k_0}m^{k_1}p^{k_2}\dots)$
    has a unique address reduces to showing that
    $(x',y') = \pi(p^{k_0-1}m^{k_1}p^{k_2}\dots)$ has a unique address.
This argument is repeated by induction, proving the result.
\end{proof}

\begin{cor}
\label{cor:dimension}
The set of uniqueness
$\mathcal U_{\be_1,\be_2}$ has positive Hausdorff dimension for any $(\be_1,\be_2)$.
\end{cor}
\begin{proof} Put $\pi'=\pi|_{U'}$. Since $U_{\beta_1,\beta_2}$ is the set
of unique addresses, the map $\pi'$ is an injection. Also, it is H\"older continuous, since $\pi$ is. Let us show
that $(\pi')^{-1}:\pi(U')\to U'$ is H\"older continuous as well.

Suppose $\underline{a}=a_1a_2\dots$ and $\underline{a'}=a_1'a_2'\dots$ with $a_i'=a_i, 1\le i\le n-1$
and $a_n\neq a_n'$. If $n=1$, then, by the above, there exists a constant $C>0$ such that
$\text{dist}(\pi(\underline{a}), \pi(\underline{a'}))\ge C$. Hence
for a general $n$ we have $\text{dist}(\pi(\underline{a}), \pi(\underline{a'}))\ge C\beta_1^{-n}$ (we assume,
as always, $\be_1>\be_2$). Since the distance between $\underline{a}$ and $\underline{a'}$ is $2^{-n}$,
we have
\[
\text{dist}(\pi(\underline{a}), \pi(\underline{a'}))\ge C\cdot \text{dist}(\underline{a}, \underline{a'})^\varkappa,
\]
where $\varkappa>0$. Hence $(\pi')^{-1}$ is H\"older continuous. The Hausdorff dimension on $\{p, m\}^{\mathbb N}$
in the usual metric coincides with the topological entropy, whence the definition of Hausdorff dimension
together with $(\pi')^{-1}$ being H\"older continuous immediately yields
$\dim_H \mathcal U_{\be_1,\be_2}\ge \dim_H \pi(U')>0$.
\end{proof}

\begin{prop}
\label{prop:entropy-drop}
For all $(\be_1,\be_2)$, the set $\mathcal U_{\be_1, \be_2}$ has no interior points.
\end{prop}

\begin{proof}
We have two cases.
Either $A$ is totally disconnected, or
    $T_{1}(A) \cap T_{-1}(A) \neq \emptyset$.
In the first case, the result is trivial.

Hence, assume that we are in the second case -- i.e., that
    $T_{1}(A) \cap T_{-1}(A) \neq \emptyset$.
Assume that $\U=\U_{\be_1,\be_2}$ has non-empty interior.
In particular, let $B$ be an open ball with $B \subset \U \subset A$.
Let $(x,y) \in T_{1}(A) \cap T_{-1}(A)$.
We know that
$A=\mathrm{cl}\left(\bigcup_{k\ge1}\bigcup_{j_1\dots j_k} T_{j_1}\dots T_{j_k}((x,y)) \right)$,
    since $A$ is the unique attractive fixed point of the iterated function system in the
    Hausdorff metric.
This implies that there exists a $j_1, j_2, \dots, j_k$ such that
    $T_{j_1} \dots T_{j_k}((x,y)) \in B \subset \U \subset A$.
As $(x,y) \notin \U$ then $T_{j_1} \dots T_{j_k}((x,y)) \notin \U$, a contradiction.
This proves the desired result.
\end{proof}

\begin{rmk}
Recall if $\be_1\be_2>2$, then the Lebesgue measure of $A_{\be_1,\be_2}$ is zero.
Consequently, the same is true for the set of uniqueness. One should expect $\mathcal U_{\be_1,\be_2}$
to have zero Lebesgue measure for all $(\be_1,\be_2)$, however even for $(\be_1,\be_2)\in \Z$ there
appears to be no easy way to prove this.
\end{rmk}

If the attractor has non-empty interior, we do not know whether the set of uniqueness can contain
an interior point of $A$; however, we have a partial result in this direction:

\begin{prop}\label{prop:uniq-boundary}
\begin{enumerate}
\item [(i)]
If $(x,y)=\pi(wm^\infty)$ or $\pi(wp^\infty)$ is in the set of uniqueness, then $(x,y)\in\partial A_{\be_1,\be_2}$.
\item [(ii)] We have $\pi(U')\subset \partial A_{\be_1,\be_2}$, where $U'$ is given by (\ref{eq:Uprime}).
\end{enumerate}
\end{prop}
\begin{proof}(i) Let $(x,y)=\pi(wm^\infty)$ (for $\pi(wp^\infty)$ the result will follow by symmetry).
Let $w=a_1\dots a_n$ and put
$d_1=\text{dist}(wm^\infty, \pi([\tilde{a_1}])$ and
$d_i=\text{dist}(wm^\infty, \pi([a_1\dots a_{i-1}\tilde{a_i}])$ for $2\le i\le n$, where, as usual,
$\tilde a=-a$. Since $\pi(C)$ is compact for any cylinder $C$, we have $d=\min_{1\le i\le n} d_i>0$.

Now suppose $\e<d$. Then $(x, y-\e)$ is not in the attractor; indeed, if it were, then
by our construction, its address would have to begin with $a_1\dots a_n$. This would mean
that to obtain $(x, y-\e)$, one or several of the subsequent $-1$ values in the address of $(x,y)$ would have to be
replaced with $1$s, which would only increase both coordinates. Therefore, there exist arbitrarily
close points in the neighbourhood of $(x,y)$ which are not in the attractor, i.e., $(x,y)$ cannot
be an interior point of $A$.

\smallskip\noindent (ii) Put
\[
d_k'=\text{dist}(\pi(p^km^\infty), \pi([m]))=\text{dist}(\pi(m^kp^\infty), \pi([p])).
\]
We know from the proof of Theorem~\ref{thm:uniq} that $d'=\inf_{k\ge1} d_k'>0$, and the rest of the argument
goes exactly like in (i), with $\e<d'$.
\end{proof}

\section{Simultaneous expansions}
\label{sec:simult}

Put
\begin{align*}
\mathcal D_{\be_1,\be_2}&=\left\{x\in\mathbb R: \exists (a_n)_1^\infty\in\{p, m\}^\mathbb N \mid
x=\sum_{n=1}^\infty a_n\be_1^{-n}=\sum_{n=1}^\infty a_n\be_2^{-n}\right\}\\
&=A_{\be_1,\be_2}\cap \{(x,y) : y=x\}
\end{align*}
(see Figure~\ref{fig:diag}).

\begin{figure}
\includegraphics[width=350pt,height=350pt]{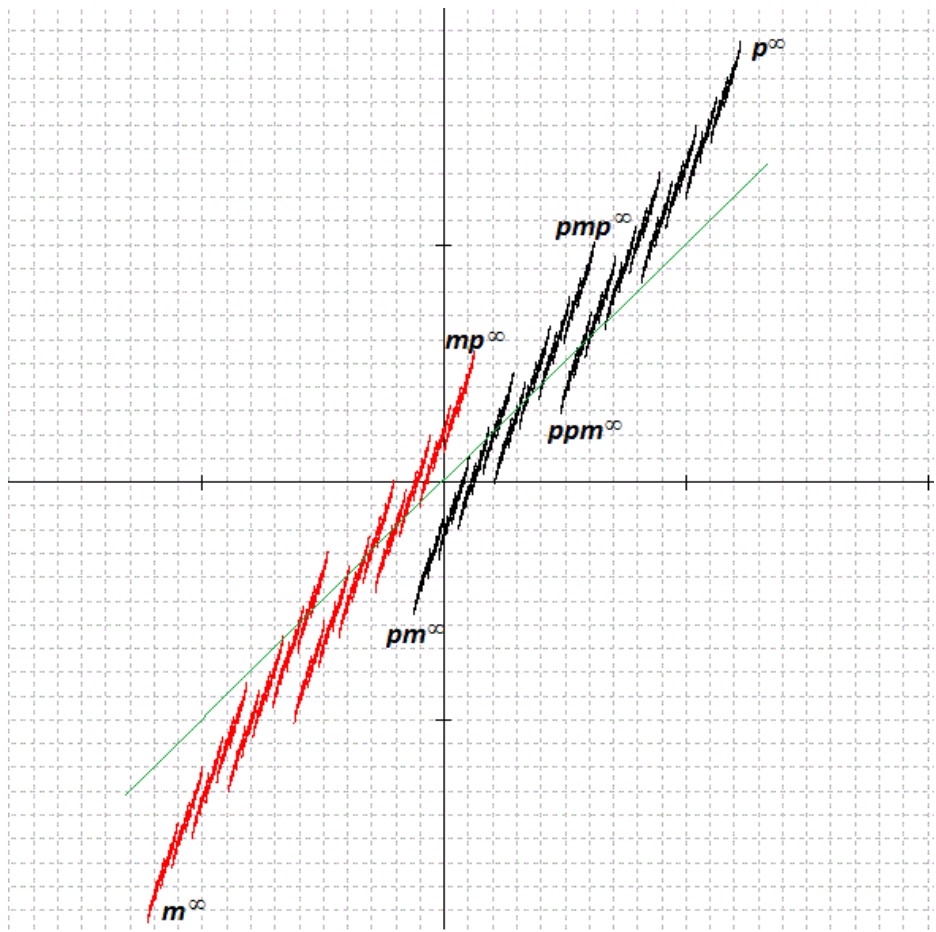}
\caption{The attractor intersecting the diagonal for $\beta_1=1.923, \beta_2 =1.754$.}
\label{fig:diag}
\end{figure}

\begin{proof}[Proof of Theorem~\ref{thm:simult}]
(i) Let $\la=\be_1^{-1}, \mu=\be_2^{-1}$ and assume $\la<\mu$.
We first claim that for any $k\ge0$ there exists a word $w\in\{p, m\}^k$ such that
$\pi(wm^\infty)$ is below the diagonal (by which we always mean the straight line $y=x$), and
$\pi(wp^\infty)$ is above it.

Note first that
that $\pi(p^\infty)=(\la/(1-\la),\mu/(1-\mu))$, and since $\la<\mu$, we have
that $\pi(p^\infty)$ lies above the diagonal. Similarly,
$\pi(m^\infty)$ lies below it -- see Figure~\ref{fig:diag}.

Proceed by induction (``bisection'') and assume the claim is true for $k=n$ and some $w$.
We will show that it is then true for $w'=wp$ or $wm$ (or both).
We have
\begin{align*}
s_\la(wmp^\infty)&=s_\la(w)-\la^{n+1}+\frac{\la^{n+2}}{1-\la}\\
&>s_\la(w)+\la^{n+1}-\frac{\la^{n+2}}{1-\la}\\
&=s_\la(wpm^\infty),
\end{align*}
in view of $\la>1/2$. Similarly, $s_\mu(wmp^\infty)>s_\mu(wpm^\infty)$.
Consider the vector from $\pi(wpm^\infty)$ to $\pi(wmp^\infty)$ given by
    \[ \pi(wmp^\infty) - \pi(wpm^\infty) = 2\left(\lambda^{n+1} - \frac{\lambda^{n+2}}{1-\lambda},
        \mu^{n+1} - \frac{\mu^{n+2}}{1-\mu}\right). \]
We see that this vector has slope
\[
\left(\frac{\mu}{\la}\right)^{n+1}\cdot \frac{2\mu-1}{1-\mu}\cdot\frac{1-\la}{2\la-1}>1,
\]
since $\la<\mu$ and the function $x\mapsto (2x-1)/(1-x)$ is strictly increasing.
Hence it would be
impossible for $\pi(wmp^\infty)$ to be below the diagonal and at the same time for
$\pi(wpm^\infty)$ to lie above it.
Now, if $\pi(wmp^\infty)$ is above the diagonal, then we put $w'=wm$;
if $\pi(wpm^\infty)$ is below the diagonal, then we put $w'=wp$;
and if both of these are true, we can choose either $w' = wm$ or $w' = wp$.

Thus, this allows us to construct a sequence of nested words $a_1\dots a_n$ such that
$\pi(a_1,a_2,\dots)$ lies on the diagonal.

\smallskip\noindent (ii)
Let us look at the bisection algorithm more closely in order to determine when we can actually choose {\em both} $wm$ and $wp$ as $w'$. 
Our aim is to construct a sequence of maps $\tau_n:[0,1]\to[0,1]$ which will keep track of all words $w$ such that $\pi(wp^\infty)$ is 
above the diagonal and $\pi(wm^\infty)$ is below it. The map $\tau_n$ turns out to be the multivalued $\beta$-transformation with
$\beta=\beta^{(n)}$, which are well understood.
Here we have that $\beta^{(n)} \uparrow \beta_2 < (\sqrt{5}+1)/2$.
The condition $\beta_2<(\sqrt5+1)/2$ implies that the number of such 
$w$ grows exponentially with $n$, which will yield the claim.

\begin{figure}
\includegraphics[width=350pt,height=350pt]{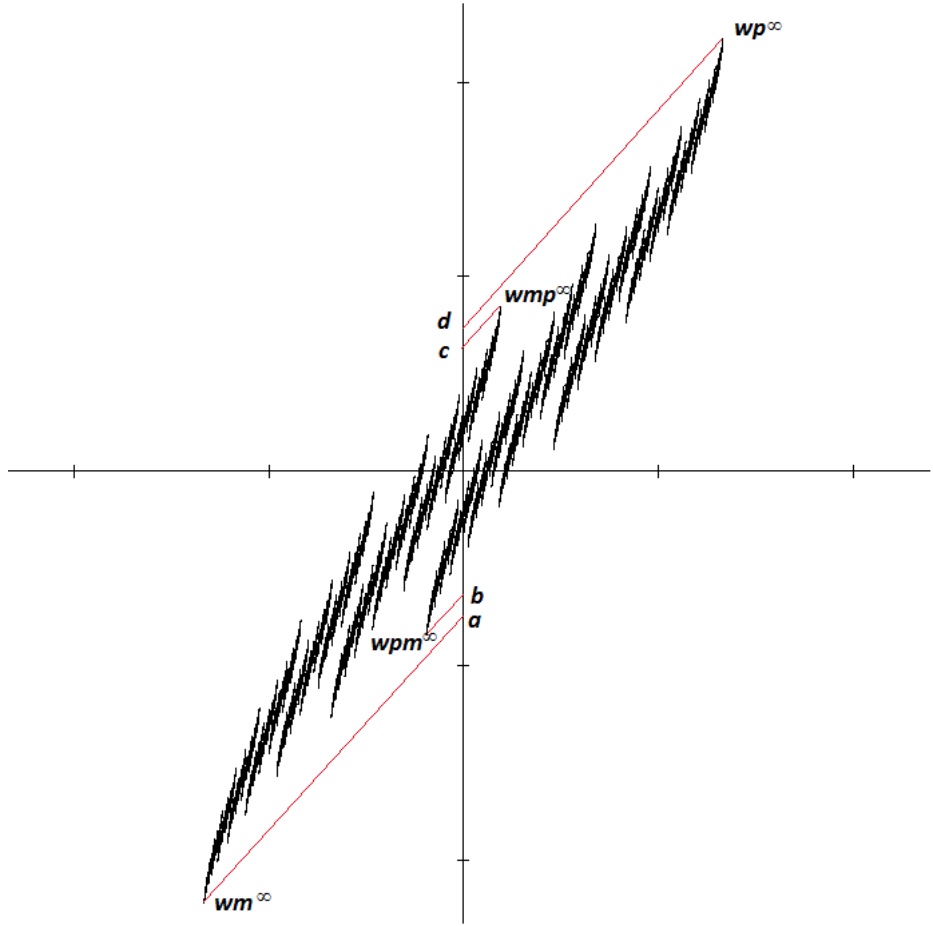}
    \caption{Projections for $\beta_1 = 1.75, \beta_2 = 1.45$}
    \label{fig:proj}
\end{figure}

Let $h$ denote
    the projection along the diagonal onto the $y$-axis, given by
    $h(x,y) = (0, y-x)$.
Put $(0,a)=h(wm^\infty), (0,b)=h(wpm^\infty), (0,c)=h(wmp^\infty)$
and finally, $(0,d)=h(wp^\infty)$ -- see Figure~\ref{fig:proj}. Let $n$ stand for the length of $w$.
A straightforward computation yields that the second coordinates of these points are respectively
\begin{align*}
a&= s_\mu(w)-s_\la(w)-\mu^{n+1}-\frac{\mu^{n+2}}{1-\mu}+\la^{n+1}+\frac{\la^{n+2}}{1-\la}, \\
b&= s_\mu(w)-s_\la(w)+\mu^{n+1}-\frac{\mu^{n+2}}{1-\mu}-\la^{n+1}+\frac{\la^{n+2}}{1-\la}, \\
c&= s_\mu(w)-s_\la(w)-\mu^{n+1}+\frac{\mu^{n+2}}{1-\mu}+\la^{n+1}-\frac{\la^{n+2}}{1-\la}, \\
d&= s_\mu(w)-s_\la(w)+\mu^{n+1}+\frac{\mu^{n+2}}{1-\mu}-\la^{n+1}-\frac{\la^{n+2}}{1-\la}.
\end{align*}
Since $1/2<\la<\mu$, we have that
$a<b<c<d$ provided $n$ is large enough. (Which we
may assume without loss of generality.) Notice that $b-a=d-c$.

We see by assumption that $a < 0$ and $d>0$.
We see that $\pi(wmp^\infty)$ is above the diagonal if and only if $c > 0$.
Hence if $c > 0$ then we can take $w' = wm$ and
    if $b < 0$ then we can take $w' = wp$.
If $b < 0 < c$, then both $w' = w m$ and $w' = wp$ are
    allowed inductive steps.

Now let $\rho_w$ denote the following affine map:
\[
\rho_w(t)=\frac{t-a}{d-a} = \frac{t-s_\mu(w)+s_\la(w)+\frac{\mu^{n+1}}{1-\mu}-
\frac{\la^{n+1}}{1-\la}}{\frac{2\mu^{n+1}}{1-\mu}-\frac{2\la^{n+1}}{1-\la}}.
\]
Put
\[
\be^{(n)}=\frac{\frac{\mu^{n+1}}{1-\mu}-
\frac{\la^{n+1}}{1-\la}}{\frac{\mu^{n+2}}{1-\mu}-\frac{\la^{n+2}}{1-\la}}\uparrow
\mu^{-1}=\be_2,\quad n\to+\infty.
\]
We have $\rho_w(a)=0, \rho_w(d)=1$ and
\begin{align*}
\rho_w(b)&=\frac{\mu^{n+1}-\la^{n+1}}{\frac{\mu^{n+1}}{1-\mu}-\frac{\la^{n+1}}{1-\la}} = 1-1/\be^{(n)}<1-\mu,\\
\rho_w(c)&=\frac{\frac{\mu^{n+2}}{1-\mu}-\frac{\la^{n+2}}{1-\la}}{\frac{\mu^{n+1}}{1-\mu}-\frac{\la^{n+1}}{1-\la}}=
1/\be^{(n)}>\mu.
\end{align*}
Note that $\rho_w(0) \in [0,1]$. We see that if $\rho_w(0) < \rho_w(c)$, then we can take $w' = wm$.
We observe that
\begin{eqnarray*}
\rho_{wm}(t) & = & \frac{t - a'}{d' - a'} \\
             & = & \frac{t-s_\mu(w m)+s_\la(wm )+\frac{\mu^{n+2}}{1-\mu}-\frac{\la^{n+2}}{1-\la}}
                   {\frac{2\mu^{n+2}}{1-\mu}-\frac{2\la^{n+2}}{1-\la}} \\
             & = & \frac{t-s_\mu(w)+s_\la(w)+\frac{\mu^{n+1}}{1-\mu}-\frac{\la^{n+1}}{1-\la}}
                   {\frac{2\mu^{n+2}}{1-\mu}-\frac{2\la^{n+2}}{1-\la}} \\
             & = & \beta^{(n)} \rho_w(t).
\end{eqnarray*}
In a similar way, if $\rho_w(0) > \rho_w(b)$ then we can take $w' = wp$,
    and \[ \rho_{wp}(t) = \beta^{(n)} \rho_w(t) + 1 - \beta^{(n)}. \]

Thus, we have a sequence of finite sets $X_n=X_n(\be_1,\be_2)$ such that $X_n=\tau_n(X_{n-1})$, where $\tau_n$
is the following multi-valued map on $[0,1]$:
\[
\tau_n(x)=\begin{cases}
\be^{(n)} x, & 0\le x<1-1/\be^{(n)},\\
\be^{(n)} x\ \  \mathrm{and} \ \ \be^{(n)} x+1-\be^{(n)},  & 1-1/\be^{(n)} \le x \le 1/\be^{(n)},\\
\be^{(n)} x+1-\be^{(n)}, & 1/\be^{(n)} < x \le 1.
\end{cases}
\]
This is a well known $\be$-expansion-generating map (with $\be=\be^{(n)}$) -- see, e.g., \cite[Section~2]{AD}.
Since $\be^{(n)}<\be_2<(1+\sqrt5)/2$, we have that for any $x_0\in(0,1-1/\be^{(n)})$, there exists
$k$ such that $\tau_k\dots \tau_1(x_0)\in (1-1/\be^{(n)}, 1/\be^{(n)})$, i.e., the trajectory of $x_0$ bifurcates
after $k$ steps. This is because $\tau_n(1-1/\be^{(n)})<1/\be^{(n)}$, in view of $(\be^{(n)})^2<\be^{(n)}+1$. This
proves that $\mathcal D_{\be_1,\be_2}$ has the cardinality of the continuum.

Furthermore, \cite[Theorem~5.2]{FS} implies that for the iterations of
a single map $\tau_n$ with $\be^{(n)}<(1+\sqrt5)/2$,
we have that no matter what $x_0\in(0,1)$, hitting the interval
$(1-1/\be^{(n)},1/\be^{(n)})$ occurs with a positive (lower) asymptotic
frequency. The argument for the sequence of maps $\{\tau_n\}$ is exactly the same, so we omit it.

Let $W_n$ denote the number of 0-1 words $w$ of length~$n$ such that $\pi(wm^\infty)$ is below the diagonal and
$\pi(wp^\infty)$ is above it. We have just shown that $W_n$ grows exponentially fast, which implies that the set
$\mathcal D_{\be_1,\be_2}\cap \{y=x\}$ has positive Hausdorff dimension
(for the same reason as in the proof of Corollary~\ref{cor:dimension}).

\smallskip\noindent
(iii) This follows from Theorem~\ref{thm:Z}.
Namely, consider in Theorem~\ref{thm:tool} the special case
    of simultaneous expansions, that is where $x_1 = x_2$,
    with the polynomial
\[ P(x) = x^8 - \frac{\beta_2^8-\beta_1^8}{\beta_2^7-\beta_1^7} x^7 +
          \frac{\beta_2^7 \beta_1^7 (\beta_2-\beta_1)}{\beta_2^7-\beta_1^7}. \]
We see that we require $|u_{-8}|, |u_{-7}| \leq 1$.
Solving for $u_{-8}$ and $u_{-7}$, we have
\begin{eqnarray*}
|u_{-8}| & = & |x_1| |b_0| (\beta_1 + \beta_2) \\
         & = & |x_1| \frac{\beta_2^7 \beta_1^7 (\beta_2 + \beta_1)}
                          {\beta_1^6 + \beta_1^5 \beta_2 + \beta_1^4 \beta_2^2 + \beta_1^3 \beta_2^3 +
                          \beta_1^2 \beta_2^4 + \beta_1 \beta_2^5 + \beta_2^6}, \\
|u_{-7}| & = & |x_1| |b_0| (\beta_1 \beta_2) \\
         & = & |x_1| \frac{\beta_2^8 \beta_1^8 }
                          {\beta_1^6 + \beta_1^5 \beta_2 + \beta_1^4 \beta_2^2 + \beta_1^3 \beta_2^3 +
                          \beta_1^2 \beta_2^4 + \beta_1 \beta_2^5 + \beta_2^6}.
\end{eqnarray*}
For $\beta_1, \beta_2 \leq 1.202\dots$,  we see that both $|b_0| (\beta_1 + \beta_2)$ and
    $|b_0| \beta_1 \beta_2$ are maximized when $\beta_1 = \beta_2 = 1.202\dots$.
This is in fact maximized for all $\beta_1, \beta_2$ where $|b_0| + |b_7| \leq 2$ at the
    exact same value, although this is not needed for the desired result.

The maximum value that $|b_0| (\beta_2 + \beta_1)$ attains with this
    restriction is approximately $1.504520168$.
This show that for all $|x_1| \leq 1/1.504520168 \approx 0.6646637388$ we have $|u_{-7}| \leq 1$.

The maximum value that $|b_0| \beta_2 \beta_1$ attains with this
    restriction is approximately $0.9047548367$.
This show that for all $|x_1| \leq 1/0.9047548367 \approx 1.105271792$ we have $|u_{-8}| \leq 1$.

Combining the two, for all $|x_1| \leq 0.664$ we have $|u_{-7}|, |u_{-8}| \leq 1$ and hence
    there exists a simultaneous expansion of $(x_1, x_1)$.
\end{proof}

\begin{rmk}
Let
\[
\widetilde{\mathcal D}_{\be_1,\be_2}=\left \{x : \exists (a_1,a_2,\dots)\in\{-1,0,1\}^\mathbb N \mid
x=\sum_{n=1}^\infty a_n\be_1^{-n}=\sum_{n=1}^\infty a_n\be_2^{-n}
\right\}.
\]
(So, the difference with $\mathcal D_{\be_1,\be_2}$ is in allowing extra zero digit.) It is shown in
\cite{KP} that $\widetilde{\mathcal D}_{\be_1,\be_2}$ has the cardinality of the continuum for all
$(\be_1,\be_2)\in (1,2)\times(1,2)$.
\end{rmk}

\section{The sets $\OO$ and $\S$}
\label{sec:O S}

We now focus our attention on the pairs $(\be_1,\be_2)$ for which the IFS satisfies the open set condition (OSC) or is
    totally disconnected.

We begin with a simple observation.
Clearly, $T_i(K)\subset K$ for $i\in\{\pm1\}$. Put $K_n=\bigcup_{|w|=n}T_w(K)$;
then $K_{n+1}\subset K_n$, and $\bigcap_{n\ge1} K_n=A$. Hence $A$ is disconnected if and
only if there exists $n$ such that $K_n$ is disconnected. (And therefore, so is $K_k$
for all $k>n$.) This immediately yields the following:

\begin{prop}\label{prop:S-open}
The set $\S$ is open.
\end{prop}
\begin{proof}Let $(\be_1,\be_2)\in\mathcal S$ and $n$ be such that $K_n$ is disconnected.
By the continuity of $T_{-1}$ and $T_1$, a sufficiently small perturbation of $(\be_1,\be_2)$ leaves
$K_n$ disconnected, whence $A$ is disconnected as well.
\end{proof}

For ease of discussion if $T_{1}(K_n^o)\cap T_{-1}(K_n^o) = \emptyset$ then
    we will say that $T_{1}(K_n)\cap T_{-1}(K_n)$ has trivial intersection.
Let $A$ be the IFS in question, and $K$ the convex hull of $A$.
We immediately see that a sufficient condition for $A$ to satisfy the OSC, or to be totally
    disconnected is if $T_1(K)$ and $T_{-1}(K)$ have trivial or empty intersection.
That is, we have

\begin{lemma}
\label{thm:K}
Let $K$ be the convex hull of $A$.
\begin{itemize}
\item If $T_{1}(K^o)\cap T_{-1}(K^o) = \emptyset$ then
    $A$ satisfies the open set condition.
\item If $T_{1}(K)\cap T_{-1}(K) = \emptyset$ then $A$ is totally disconnected.
\end{itemize}
\end{lemma}

Here $K^o$ is the interior of $K$.
Although these requirements are sufficient, they are not necessary.
This is because $K$ is a extreme overestimate for the shape of $A$.

In Figure~\ref{fig:S1} we have shown those $(\beta_1, \beta_2)$ such that they
    satisfy the hypothesis of Lemma~\ref{thm:K}.

\begin{figure}
\includegraphics[width=300pt,height=300pt]{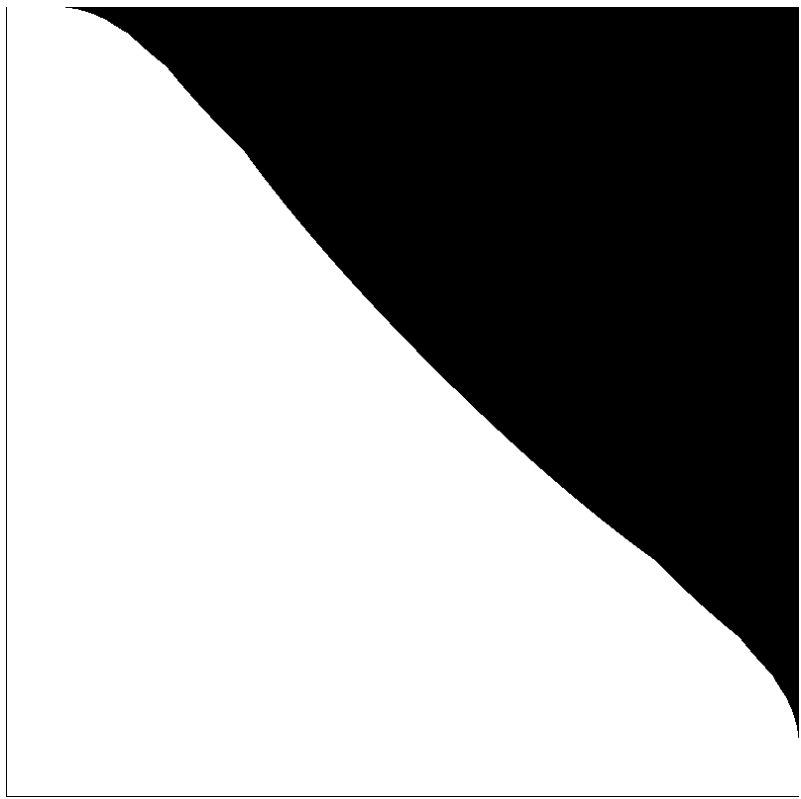}
\caption{Points known to be in $\S$ (black).
         (Level 1 approximation)}
\label{fig:S1}
\end{figure}

This curve is the same curve, after translation of notation, to that found by Solomyak \cite{Sol}
    using somewhat different techniques.
This will be shown in Theorem~ \ref{thm:Solomyak}. A precise description of this
curve is given in Theorem~\ref{thm:S struct}.

The idea of approximating $A$ by a simple set $K$ can be generalized.
Recall that for $w \in \{p, m\}^*$ that $K_w = T_w(K)$ and we define
    $K_n = \bigcup_{|w| = n} K_w$.
A immediate, and profitable, generalization of Lemma~\ref{thm:K} gives

\begin{lemma}
\label{lem:Kn}
Let $K_n$ be as above.
\begin{itemize}
\item If $T_{1}(K_n^o)\cap T_{-1}(K_n^o) = \emptyset$ then
    $A$ satisfies the open set condition.
\item If $T_{1}(K_n)\cap T_{-1}(K_n) = \emptyset$ then $A$ is totally disconnected.
\end{itemize}
\end{lemma}

This can of course to be done for any set that contains $A$ as a subset.
An advantage of these $K_n$ is that $K_n \to A_{\be_1, \be_1}$ in the
    Hausdorff metric.

In Figure~\ref{fig:S40} we have given the approximations of $\S$ based on
    $K_{40}$.
We will call an approximation of $\S$ using Lemma~ \ref{lem:Kn} with a
    particular $K_n$, a  {\em level $n$ approximation}.

In Theorem~\ref{thm:K vert} we gave a precise description of the vertices of
    $K$.
We can now determine for which
    $\beta_1, \beta_2$ we satisfy the conditions of Lemma~\ref{thm:K} and, to some extent, \ref{lem:Kn}.

Let $M_k$ be the line connecting $m^k p^\infty$ and $m^{k+1} p^\infty$,
    and similarly $P_k$ for $p^k m^\infty$ and $p^{k+1} m^\infty$.
(See Figure~\ref{fig:K}.)

\begin{lemma}
For each $\beta_1 > \beta_2$ there exists $k$ such that the segment
    $T_1(M_k)$ crosses the $y$-axis.
\end{lemma}

It should be noted that that this $k$ may not be unique, as it is possible that
    $T_1(m^k p^\infty)$ is on the $y$-axis.
In this case we would say that both $k-1$ and $k$ satisfy this criterion.

\begin{proof}
We see that $\pi(pm^\infty)$ lies to the left of the $y$-axis, and that
    $\pi(p^\infty)$ lies to the right.
This, combined with the fact that the $M_k$ form a decreasing (with respect of the
    $y$-coordinate) sequence of intervals proves the result.
\end{proof}

We will denote this $k := k(\beta_1, \beta_2)$.

\begin{lemma}
\label{lem:SOcomp}
Assume $\beta_1 > \beta_2$ and let $k := k(\beta_1, \beta_2)$.
Then
\begin{itemize}
\item If $T_1(M_k)$ is below the point $(0,0)$ then $T_1(K) \cap T_{-1}(K)=\varnothing$;
\item If $T_1(M_k)$ goes through the point $(0,0)$ then $T_1(K) \cap T_{-1}(K)$ has trivial, but
    non-empty intersection;
\item If $T_1(M_k)$ is above the point $(0,0)$ then $T_1(K) \cap T_{-1}(K)$ has non-trivial and
    non-empty intersection.
\end{itemize}
\end{lemma}

We see that the first case gives a sufficient condition for $(\beta_1, \beta_2) \in \S$. Also,
the first case combined with the second one gives criteria for when $(\beta_1, \beta_2) \in \OO$.
Unfortunately the final case does not yield anything useful about $(\beta_1, \beta_2)$ --
it only indicates that the level of approximation we are using is insufficient to come to a conclusion.

\begin{proof}
This follows from the symmetry of $T_1(K)$ and $T_{-1}(K)$ and the fact that
    $\beta_1 > \beta_2$.
See for example Figure~\ref{fig:T1(K)}.
\end{proof}

\begin{figure}
\includegraphics[width=100pt,height=100pt]{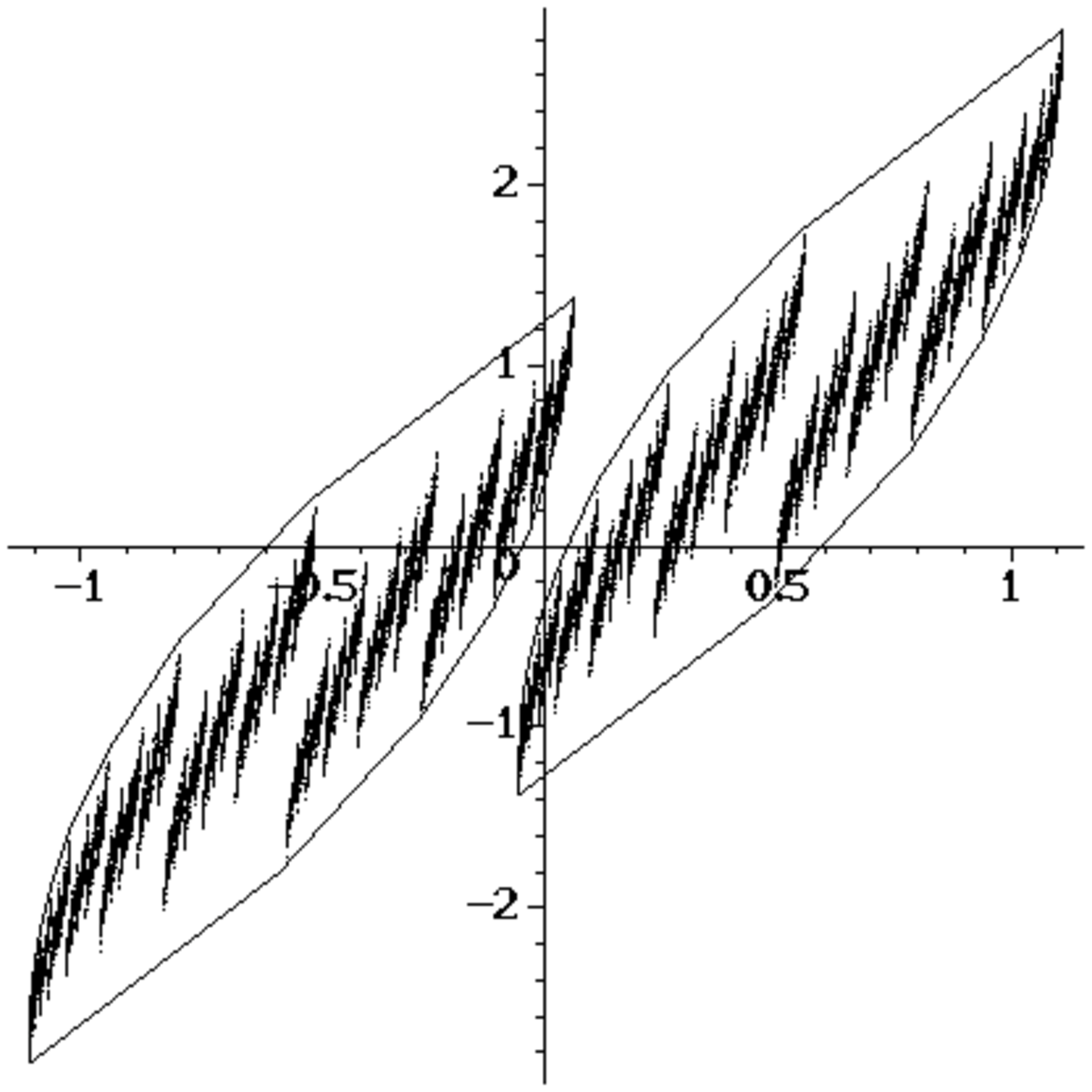}
\includegraphics[width=100pt,height=100pt]{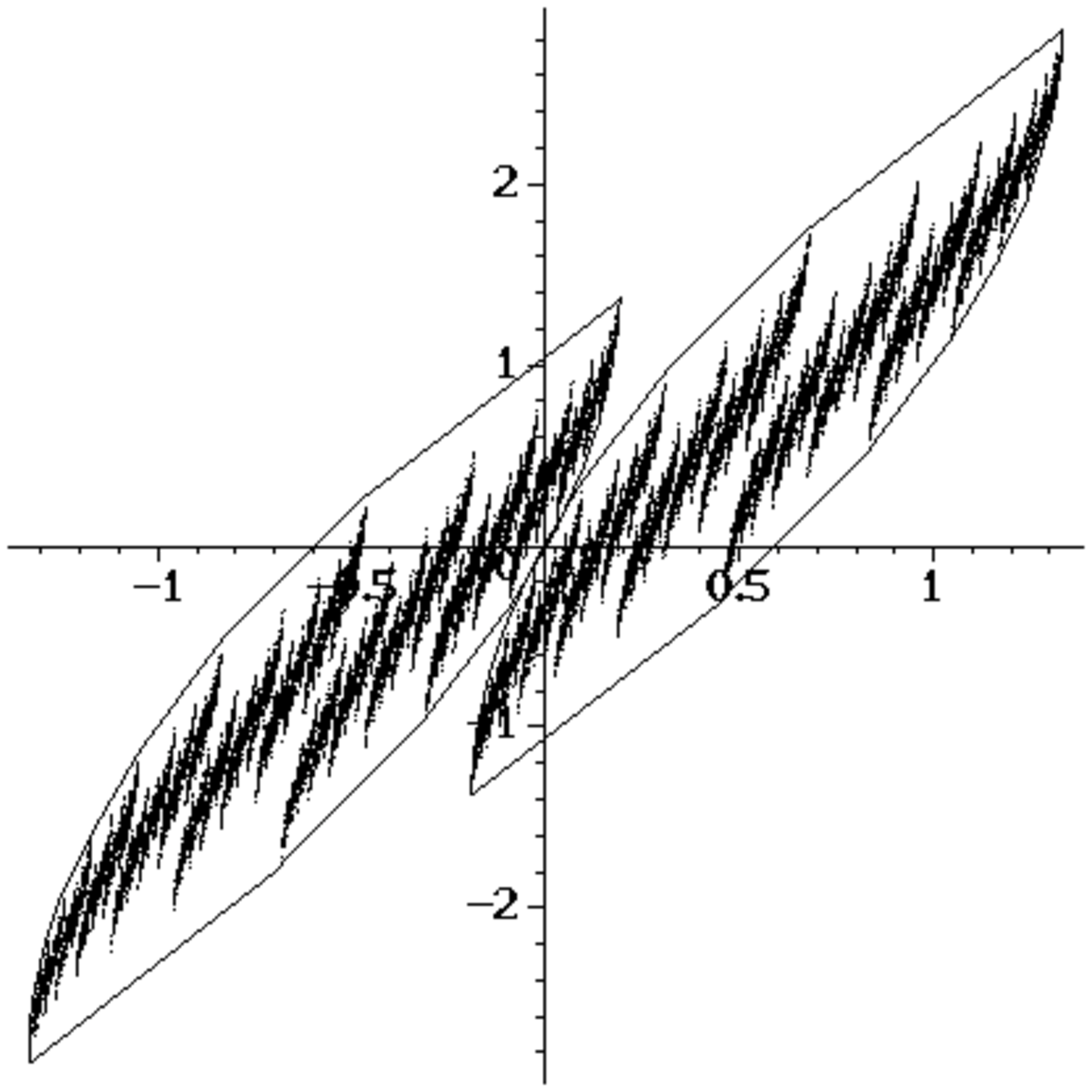}
\includegraphics[width=100pt,height=100pt]{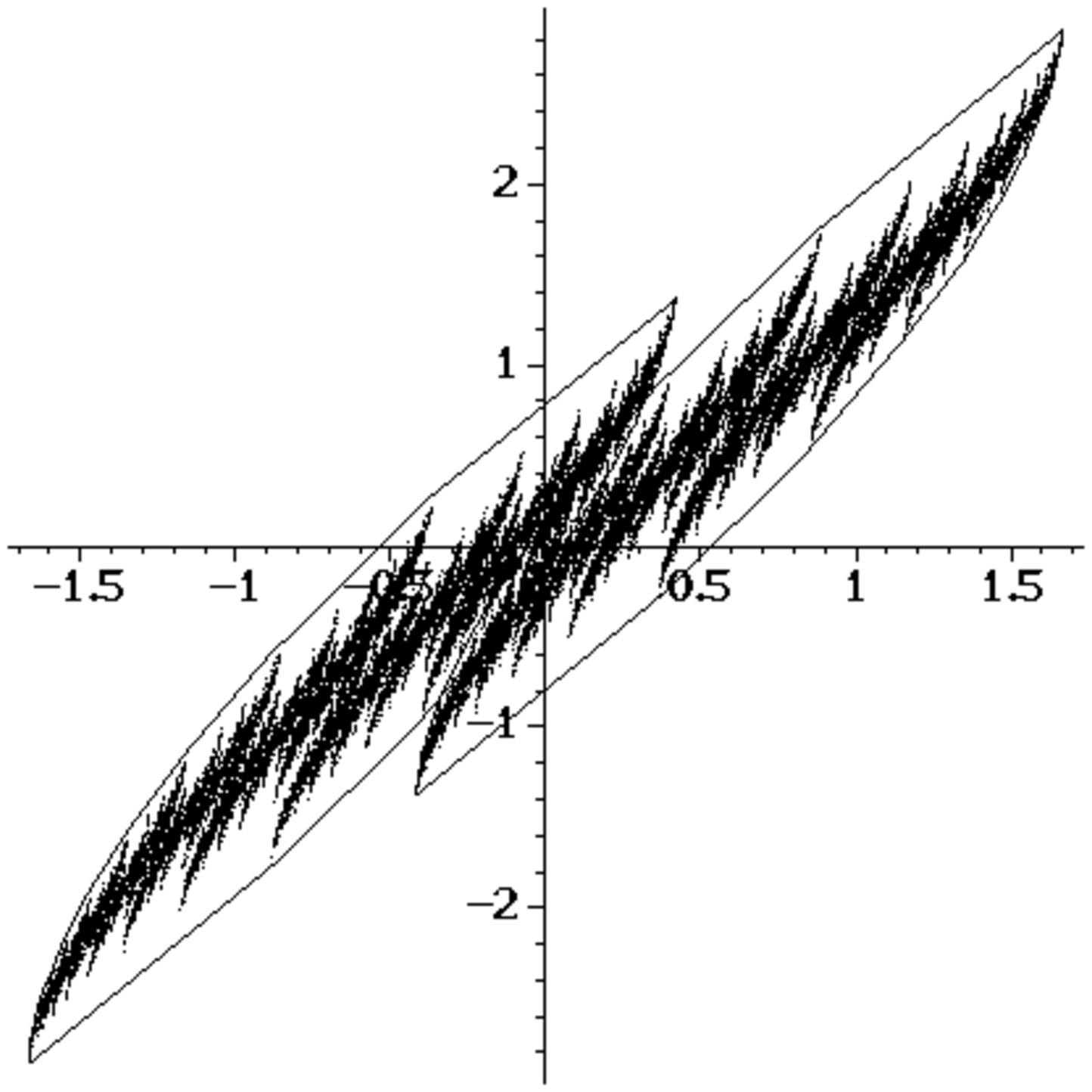}
\caption{Level 1 approximation for $\beta_1 \approx 1.9, 1.75$ and $1.6$ with $\beta_2 = 1.35$.}
\label{fig:T1(K)}
\end{figure}

Using this, we can now give criteria for a point $(\beta_1, \beta_2)$ to be in a level~1 approximation.

Define
\[S_1 = \{(\beta_1, \beta_2) \mid T_1(K) \cap T_{-1}(K)\
    \text{has trivial but non-empty intersection}\}.
\]

\begin{thm}
\label{thm:S struct}
Let $P_k(x) = x^{k+1} - 2 x^k + 2$.
Let $(\beta_1^{(k)}, \beta_2^{(k)})$ be the two roots of $P_k$
    between $1$ and $2$, with $\beta_1^{(k)} < \beta_2^{(k)}$.
Then
\begin{enumerate}
\item[(i)] For $k \geq 4$ we have
      $(\beta_1^{(k)}, \beta_2^{(k)}), (\beta_2^{(k)}, \beta_1^{(k)}) \in S_1$.
\label{en:a}
\item[(ii)] For $k \geq 4$, let $\beta_1^{(k)} \leq \beta_1 \leq \beta_1^{(k+1)}$ and
                          $\beta_2^{(k)} \leq \beta_2 \leq \beta_2^{(k+1)}$
      satisfy
        \begin{equation}
            P_k(\beta_1) P_{k+1}(\beta_2) - P_{k+1}(\beta_1) P_{k}(\beta_2) = 0.
            \label{eq:inside}
        \end{equation}
      Then $(\beta_1, \beta_2), (\beta_2, \beta_1) \in S_1$.
\label{en:b}
\item[(iii)] Let $\beta_1^{(4)} \leq \beta_1 < \beta_2 \leq \beta_2^{(4)}$
      satisfy
        \begin{equation}
            P_3(\beta_1) P_{4}(\beta_2) - P_{4}(\beta_1) P_{3}(\beta_2) = 0.
            \label{eq:inside4}
        \end{equation}
      Then $(\beta_1, \beta_2), (\beta_2, \beta_1) \in S_1$.
\label{en:c}
\item[(iv)] We have $\beta_2^{(k)}\to 1,\ \beta_1^{(k)}\to2$ as $k\to+\infty$.
\label{en:d}
\end{enumerate}
\end{thm}

\begin{figure}
\includegraphics[width=300pt,height=300pt]{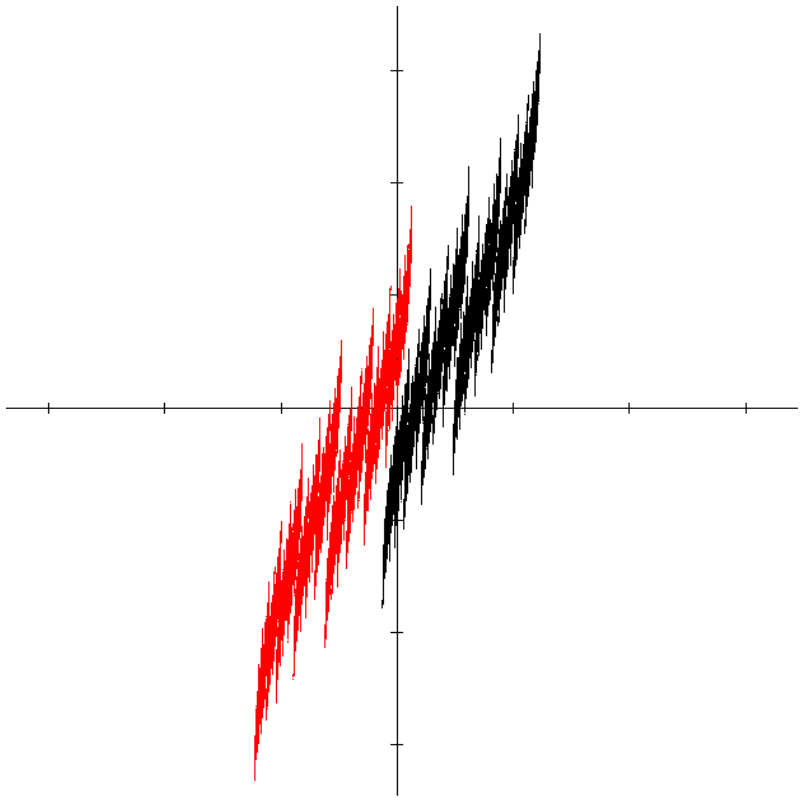}
\caption{``Just touching'': we have $T_{-1}(A)\cap T_1(A)=\{(0,0)\}$ for
$\beta_1\approx1.81618, \beta_2\approx1.30022$ being roots of $x^5-2x^4+2$.
Furthermore, here $T_{-1}(K)\cap T_1(K)=\{(0,0)\}$ as well.}
\label{fig:touching}
\end{figure}

\begin{proof}
(i)
Assume that $T_{1}(K) \cap T_{-1}(K)$ has trivial but non-empty intersection.
This implies that one of the edges or corners of $T_1(K)$ contains $(0,0)$.
Assume first that $(0,0)$ is a corner;
then we have that $T_1(\pi(m^k p^\infty)) = (0,0)$.
This implies
    \[
    \beta_1^{k+1} - 2 \beta_1^k + 2 =
       \beta_2^{k+1} - 2 \beta_2^k + 2 =  0,
    \]
which corresponds to the point $(\beta_1^{(k)}, \beta_2^{(k)})$.
It is worth observing that the above equation has no solutions for $k \leq 3$.
This resulting in the interesting consequence that the first, second, third and fourth level approximations
    are all the same.

(ii)
Next assume that, instead of a corner, it is a line that goes through $(0,0)$.
We see that the line $T_1(M_k)$ will intersect the point $(0,0)$ if the line
    from $T_1(\pi(m^k p^\infty))$ to $T_1(\pi(m^{k+1} p^\infty))$ goes through $(0,0)$.
Letting $(x_k, y_k) = T_1(\pi(m^k p^\infty))$ and
	$(x_{k+1}, y_{k+1}) = T_1(\pi(m^{k+1} p^\infty))$,
    we see that the $y$-intercept of the line through these points is
    \[ \frac{x_k y_{k+1} - y_k x_{k+1} }{x_{k+1}- x_k}. \]
This will equal zero when
\[ 0 = x_k y_{k+1} - y_k x_{k+1}. \]
Evaluating the above equation at $\beta_1$ and $\beta_2$ gives equation \eqref{eq:inside}.
It is worth observing that the line segment between $(x_k, y_k)$ and $(x_{k+1}, y_{k+1})$
    will only cross the $y$-axis if these two points are on the opposite sides of the axis.
This implies that $\beta_1^{(k)} < \beta_1 < \beta_1^{(k+1)}$ and
                  $\beta_2^{(k)} < \beta_2 < \beta_2^{(k+1)}$.

(iii)
Similar to (ii).

(iv)
Finally, the equation $x^k=2(x^{k-1}-1)$ becomes $t^k=t-\frac12$ for $t=x^{-1}$. It is clear
from the graphs of the left and right hand sides that the sequence of smaller real roots, $\rho_k$,
is decreasing, while the sequence of larger real roots, $\rho_k'$, is increasing. Therefore,
$\rho_k^k\to0$, whence $\rho_k\to\frac12$, which is equivalent to $\beta_1^{(k)}\to2$ as $k\to+\infty$.
On the other hand, $\rho_k'\to1$, since it is always smaller than 1 and cannot tend to $\kappa<1$, since
in that case $\kappa$ must be equal to $\frac12$ as well, which is impossible. Hence
$\beta_2^{(k)}\to1$.
\end{proof}

Figure~\ref{fig:touching} illustrates the above theorem for $\beta_i=\beta_i^{(4)}, i=1,2$.

\begin{proof}[Proof of Corollary~\ref{cor:3.129}]
Consider the curves
    $P_{k}(\beta_1)P_{k+1}(3-\beta_1+t) - P_{k+1}(\beta_1)P_{k}(3-\beta_1+t)=0$.
Solving for the local maxima of these (with respect to $t$),
    we see that the local maximum for $k = 4$ is maximal, and obtains a
    value of
    \[ t       = 0.1294734398566760176850196318981206812538310097982 \dots \]
when
    \[ \beta_1 = 1.2356028604456261036844313175875156433117845240595 \dots\]
Precise algebraic quantities can be given in terms of the roots of a degree~$36$ polynomial, which we omit.

It was shown in \cite[Theorem 2.3]{Sol} that all neighbourhoods
    of $(\beta_1^{(k)},\beta_2^{(k)})$ contain a point that is
    not in $\S$.
Taking $k = 5$ proves the second inequality.
\end{proof}

It is worth observing that B.~Solomyak \cite{Sol} came at this through a different construction.
Solomyak first considered the function
\begin{equation}
h_k^{(t)} = 1 - x - \dots - x^{k-1} + t x^k + x^{k+1} + x^{k+2} + \dots.
\end{equation}
Following \cite{Sol}, put
\[
\mathcal B_{[-1,1]}=\left\{1+\sum_{n=1}^\infty a_nz^n\mid a_n\in[-1,1]\right\}.
\]
For $f\in\mathcal B_{[-1,1]}$ let $\xi_1(f)\le \xi_2(f)\le\dots$ denote the positive
zeroes of $f$ ordered by magnitude and counted with multiplicity. Let
\[
\phi:\gamma\to\min\{\xi_2(f) : f\in\mathcal B_{[-1,1]}, \ f(\gamma)=0\}.
\]
By \cite[Proposition~2.2]{Sol}, the function $\phi$ is well defined. Furthermore,
let $\alpha_2\approx 0.649138$ be the positive zero of $2x^5-8x^2+11x-4$. By the same Proposition,
for all $\gamma \in [1/2, \alpha_2]$ there exists a unique function $h_k^{(t)}$ such that
    $h_k^{(t)}(\gamma) = h_k^{(t)}(\phi(\gamma)) = 0$.
If $\gamma < \lambda < \phi(\gamma)$, then $(1/\gamma, 1/\lambda) \in \S$.

\begin{thm}
\label{thm:Solomyak}
The curve given by $(\gamma, \phi(\gamma))$ is the same as the level-1 approximation of $\S$
    given by Theorem~\ref{thm:S struct}.
\end{thm}

\begin{proof}
We note a few things.
\begin{itemize}
\item If $t = -1$ then $h_k^{(t)}(1/\beta) = 0$ if and only if $P_{k-1}(\beta) =0$.
\item If $t = 1$ then $h_k^{(t)}(1/\beta) = 0$ if and only if $P_{k}(\beta) =0$.
\end{itemize}
Hence the corners of this curve are the same as the corners of the curve~$\mathcal S$.

Let $x_k = s_\mu(p m^k p^\infty)$ and
and $y_k = s_\la(p m^k p^\infty)$. 
We showed that if $T_1(K)$, the first level convex approximation of
    $A$ ``just touches'' $T_{-1}(K)$ then
\begin{equation}
\label{eq:touch}
x_{k+1} y_k - y_{k+1} x_k = 0.
\end{equation}
Furthermore, $x_k$ will be on one side of the axis, and $x_{k+1}$ will be
    on the other.
Let
\begin{equation}
\label{eq:t}
t = 2 \cdot \frac{x_{k+1}}{x_{k+1} - x_{k}} -1.
\end{equation}
We see that if $x_k = 0$ (i.e. the corner of $K$, $(x_k, y_k) = (0,0)$)
    then $t = -1$.
Furthermore, if $x_{k+1} = 0$ then $t = 1$.
Hence $t$ ranges between $-1$ and $1$.
This implies that
\begin{equation}
\label{eq:t2}
\frac{t+1}{2} x_{k} = \frac{t-1}{2} x_{k+1}.
\end{equation}
Using this in equation \eqref{eq:touch} gives
\begin{eqnarray*}
0 & = & x_{k+1} y_k - y_{k+1} x_k \\
  & = & \frac{t+1}{2} x_{k+1} y_k -\frac{t+1}{2}  y_{k+1} x_k \\
  & = & \frac{t+1}{2} x_{k+1} y_k -\frac{t-1}{2} x_{k+1} y_{k+1} \\
  & = & \frac{t+1}{2} y_k -\frac{t-1}{2} y_{k+1}.
\end{eqnarray*}
It is worth noting that the values when $t+1 = 0$ and $x_n = 0$ are when
    the vertices of $K$ touch $(0,0)$,
   and hence not actually attained when it is the interior of the
    edge that meets $(0,0)$.
Hence the division and multiplication of $0$ is not a problem.
We notice that the equation $\frac{t+1}{2} y_k -\frac{t-1}{2} y_{k+1}$
    equals zero if
\begin{eqnarray*}
0 & = &  1/\beta_2 - 1/\be_2^2 - \dots - 1/\be_2^{k+1} + t /\be_2^{k+2}
          + 1/\be_2^{k+3} + 1/\be_2^{k+4} +  \dots  \\
  & = & h_{k+1}^{(t)}(1/\be_2)
\end{eqnarray*}
A similar argument shows that $h_{k+1}^{(t)}(1/\be_1) = 0$, as required.
\end{proof}

Consider a finite word $w \in \{p, m\}^n$.
Recall that $K_w = T_{w}(K)$.
By our previous notation, $K_n = \bigcup_{|w|=n} K_w$.

To check if $T_1(K_n) \cap T_{-1}(K_n)$ has empty, or trivial intersection, it suffices to
    check $T_1(K_w) \cap T_{-1}(K_{w'})$ for all words $w, w' \in \{p, m\}^n$.
To improve the efficiency of this search, we observe that
    if $T_1(K_w) \cap T_{-1}(K_{w'})$ is empty or trivial, then for all
    words $w_0, w_0'$ we have that
   $T_1(K_{w w_0}) \cap T_{-1}(K_{w' w_0'})$ is empty or trivial.

This allows us to improve the efficiency of the search.

We again remark that the level~1 approximation (using $K_1$) is the same as that found
    in \cite{Sol}.
In fact, this is the same for levels~$2, 3$ and $4$ as well.
At level~$5$ additional points are discovered to be in $\S$ that were not provable before.
(See Figure~\ref{fig:S5}.)
We could, if necessary, construct curves much like Theorem~\ref{thm:S struct}.
This trend continues as we increase to higher level approximations.
(See Figure~\ref{fig:S40}.)

\begin{figure}
\includegraphics[width=300pt,height=300pt]{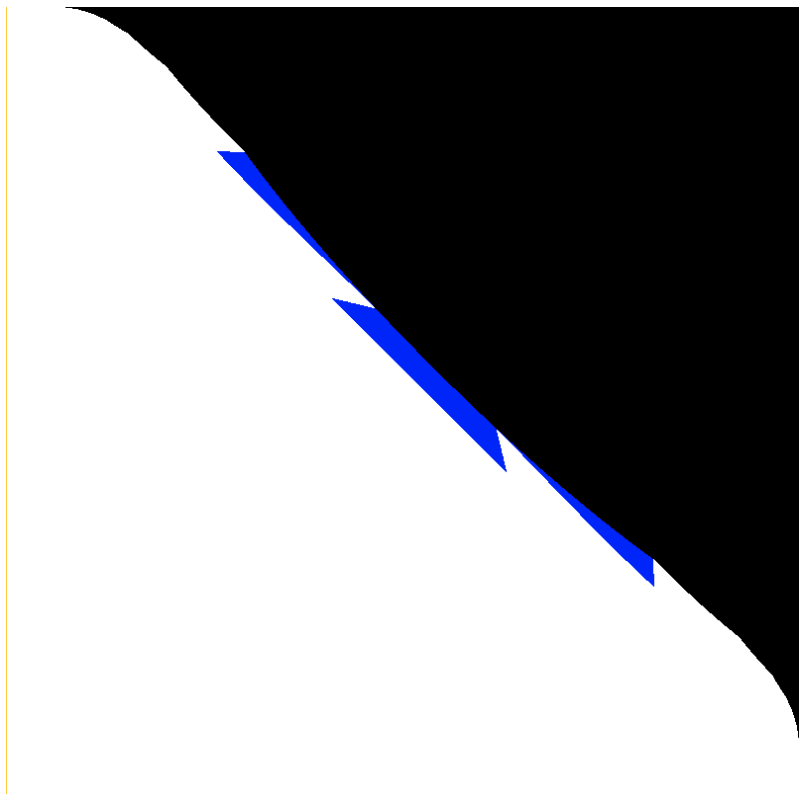}
\caption{Points in $\S$.  Those in blue come from the level~5 approximation.}
\label{fig:S5}
\end{figure}

One might conjecture, when looking at the initial pictures produced that
    all of our curves coming from a level $n$ approximation are connected.
If this were true, then this would imply that $\S$ was connected.
It turns out, rather surprisingly, that this is not the case.
At level 14 we have an occurrence of an island that is not connected to the
    main body of the curve, (see Figure~\ref{fig:island}).
More surprisingly, as we show in Section~\ref{sec:island}, this is not an artifact of our
    choice of approximations of $A$.
This is in fact a legitimate island of $\S$ that is disconnected from the main body.
This proves that $\S$ is not connected, and hence the {\em connectedness locus}
$\mathcal N=\S^c$ studied in detail in \cite{Sol} is not simply connected.

\begin{figure}
\includegraphics[width=300pt,height=300pt]{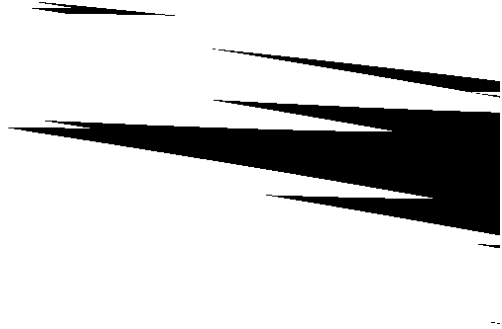}
\caption{Level 14 approximation of $\S$,
    $\beta_1 \in [1.32025, 1.35275]$,
    $\beta_2 \in [1.6306,  1.6631]$}
\label{fig:island}
\end{figure}

\section{$\S$ is not connected.}
\label{sec:island}

In Section~\ref{sec:O S} we gave a technique to show that a point $(\beta_1, \beta_2)$ corresponded
    to a totally disconnected set $A$.
Using this technique, we observed at level~14, that the approximation to $\S$ was not connected
    (see Figure~\ref{fig:island}).

In this section we will prove that this region is indeed in a separate connected component
with respect to the rest of $\S$. Namely, in Figure~\ref{fig:island}
we see a chevron shaped object $C$ which is disconnected from the main body of the
    approximation of $\S$. A significant part of our proof is computer-assisted.

First, we need to show that there exists a point in $C$ which is provably in $\S$.
A quick computer check yields
    $(1.335438104, 1.646743824) \in C \subset \S$.

To prove that $C$ is separate from the main body of $\S$ we will
    give six path connected regions, $R_{w_1}, \dots, R_{w_6}$, all disjoint from $\mathcal S$, such that
    $R_{w_1}$ overlaps with $R_{w_2}$, which in turn overlaps with
    $R_{w_3}$, and so on, where finally $R_{w_6}$ overlaps with the original
    set $R_{w_1}$.
These overlapping sets will surround $C$ -- see Figure~\ref{fig:sharks}.
\begin{figure}
\includegraphics[width=300pt,height=300pt]{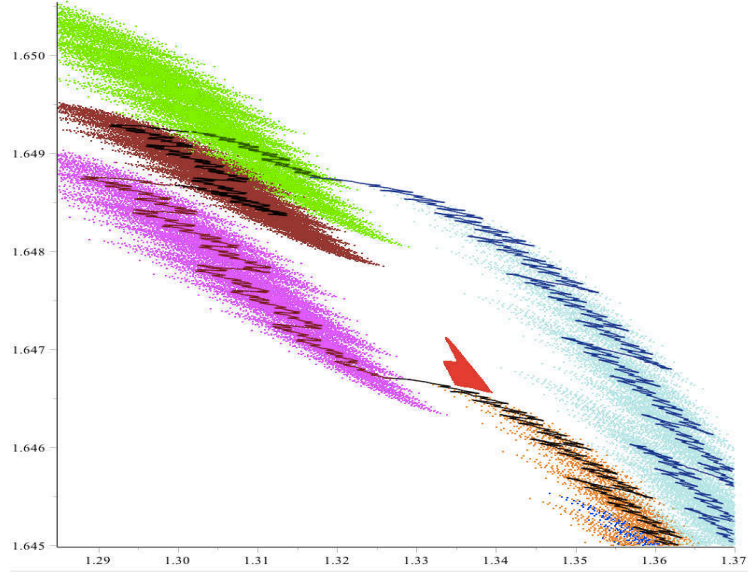}
\caption{The chevron $C$ (red) and $R_{w_1}$ (green), $R_{w_2}$ (turquoise), $R_{w_3}$ (blue),
$R_{w_4}$ (orange), $R_{w_6}$ (magenta) and $R_{w_6}$ (brown),
    along with overlapping continuous paths}
\label{fig:sharks}
\end{figure}

We need a criterion for a pair $(\be_1,\be_2)$ not to lie in $\mathcal S$.
As usual, $m$ stands for $-1$, and $p$ for $1$.
We will also use $z$ for $0$.

\begin{lemma}
\label{lem:not O}
If $\beta_1$ and $\beta_2$ are distinct roots of $P \in \mathbb{Z}[x]$
    with the coefficients of $P$ restricted to $\{p, z, m\}$ then
    $(\beta_1, \beta_2) \notin \S$.
\end{lemma}

\begin{proof}
Let $P(x) = a_n x^n + \dots + a_0$ with $a_i \in \{-1, 0, 1\}$.
Write $2 P(x)  = P_+(x) - P_-(x)$ with
    $P_+(x) = a^+_n x^n + \dots + a^+_0$ with $a_i \in \{-1, 1\}$ and
    $P_-(x) = a^-_n x^n + \dots + a^-_0$ with $a_i \in \{-1, 1\}$.
As $P(\beta_1) = P(\beta_2) = 0$ we have that
    $P_+(\beta_1) = P_-(\beta_1)$ and $P_+(\beta_2) = P_-(\beta_2)$.

Notice that
\begin{align*}
s_{1/\beta_1} ((a^+_n a^+_{n-1} \dots a^+_0)^\infty )
    & = P_+(\beta_1) (1/\beta_1^{n+1} + 1/\beta_1^{2(n+1)} + \dots) \\
    & = P_-(\beta_1) (1/\beta_1^{n+1} + 1/\beta_1^{2(n+1)} + \dots) \\
    & = s_{1/\beta_1} ((a^-_n a^-_{n-1} \dots a^-_0)^\infty ).
\end{align*}
A similar result holds for $1/\beta_2$ which gives us that
\[ \pi((a^+_n a^+_{n-1} \dots a^+_0)^\infty)  =
    \pi((a^-_n a^-_{n-1} \dots a^-_0)^\infty). \]

As $a_n \neq 0$ we see that $a_n^+ \neq a_n^-$ and hence
\[ \pi((a^+_n a^+_{n-1} \dots a^+_0)^\infty)
 = \pi((a^-_n a^-_{n-1} \dots a^-_0)^\infty) \in T_1(A)\cap T_{-1}(A).\]

This give that $A$ is connected, and hence $(\beta_1, \beta_2) \not\in \S$.
\end{proof}

\begin{rmk}
An essentially identical proof holds if $1/\beta_1$ and $1/\beta_2$ are two distinct
    roots of a power series with coefficients $\{-1, 0, 1\}$.
\end{rmk}

We next need a result of Odlyzko and Poonen \cite[Lemma~4.1]{OdlyzkoPoonen93}:

\begin{lemma}\label{lem:odl}
Let $Y$ be a topological space.
Suppose $f:\{0, 1\}^\BbN \to Y$ is a continuous map such that
    \[ f([w0]) \cap f([w1]) \neq \emptyset \]
for all $w \in \{0, 1\}^*$.
Then the image of $f$ is path connected.
\end{lemma}

Recall that $[i_1\dots i_k]$ stands for the cylinder
$\{a_j\}_{j=1}^\infty \subset \{0,1\}^\mathbb N$ such that
$a_j = i_j$ for $j = 1,2, \dots, k$. Lemma~\ref{lem:odl}
can be easily generalized to the space $\{p, z,  m\}^\BbN$:

\begin{lemma}
\label{lem:OP}
Let $Y$ be a topological space.
Suppose $f:\{p, z, m\}^\BbN \to Y$ is a continuous map such that
\begin{eqnarray*}
&& f([wz]) \cap f([wp]) \neq \emptyset \\
&& f([wm]) \cap f([wp]) \neq \emptyset \\
&& f([wm]) \cap f([wz]) \neq \emptyset
\end{eqnarray*}
for all $w \in \{p, z, m\}^*$.
Then the image of $f$ is path connected.
\end{lemma}

The proof is a simple variation of the result of Odlyzko and Poonen.
We provide it here for completeness.
\begin{proof}
This is in essence a bisection method.
Given two infinite words $w = a_1 a_2 a_3 \dots$ and $w' = b_1 b_2 b_3 \dots$,
    we define the usual metric by
    $\mathrm{dist}(w, w') = 1/2^k$ where
    $a_i = b_i$ for $i = 1, \dots, k-1$ and $a_k \neq b_k$.
If no such $k$ exists, then $w = w'$ and  $\mathrm{dist}(w, w') =0$.
Given two points $x_0 = f(w_0)$ and $x_1 = f(w_1)$, we construct two new
    words $w_{1/2}$ and $w_{1/2}'$ such that
\begin{itemize}
\item $f(w_{1/2}) = f(w_{1/2}')$,
\item $\mathrm{dist}(w_0, w_{1/2}) < \mathrm{dist}(w_0, w_1)$,
\item $\mathrm{dist}(w_{1/2}', w_1) < \mathrm{dist}(w_0, w_1)$.
\end{itemize}
To do this we let $w$ be the common prefix of $w_0$ and $w_1$ so that
    $w_0 = w a_0 v_0$ and $w_1 = w a_1 v_1$ with $a_0 \neq a_1$.
We then find $w_{1/2} \in [w a_0]$ and $w_{1/2}' \in [w a_1]$ so that
    $f(w_{1/2}) = f(w_{1/2}') \in f([w a_0]) \cap f([w, a_1])$.
Such a point exists by assumption.

We now induct on this construction to find points $x_{1/4}$ and $x_{3/4}$
and then $x_{1/8},x_{3/8},x_{5/8}$ and $x_{7/8}$ and so on.
We notice by the continuity of $f$ and the fact the distances between
    adjacent points go to $0$ in the limit, then this construction will
    define a continuous path in the image of $f$.
\end{proof}

Let $v \in \{p, m, z\}^*$ be a finite word of length $n$.
Furthermore, assume that $v_1 \neq z$.
Define $P_v(x) = P(x) = v_1 x^{n-1} + \dots + v_n$.
If $\beta_1, \beta_2$ are distinct roots of $P$ then we see from Lemma~\ref{lem:not O} that
    $(\beta_1, \beta_2) \notin \S$.

Let $\beta_1^+, \beta_1^-, \beta_2^+, \beta_2^-$ be distinct roots of
    the rational function $P(x) \pm \frac{1}{x-1}$, assuming that they exist.
Let $I_1 = [\beta_1^{\pm}, \beta_1^{\mp}]$ and $I_2 = [\beta_2^{\pm}, \beta_2^{\mp}]$.
Let $f(x) \in \left\{ \sum_{i=1}^\infty w_i x^{-i} : w \in \{p, m, z\}^\BbN \right\}$.
We see that if $|f'(x)| < |P'(x)|$ for all $x \in I_1$, then
    $P(x) + f(x)$ will have a unique root in $I_1$.
We will denote this root $\beta_1^{(w)}$.
Similarly, if $|f'(x)| < |P'(x)|$ for all $x \in I_2$, then
    $P(x) + f(x)$ will have a unique root in $I_2$, which we will denote $\beta_2^{(w)}$.

We see that if $|P'(x)| > 1/(x-1)^2$ for all $x \in I_1$ and $x \in I_2$, then there
    will be well defined roots $\beta_1^{(w)}$ and $\beta_2^{(w)}$ for all
    $w \in \{p, m, z\}^\BbN$.

We will call the existence of $\beta_1^{\pm}$, $\beta_2^{\pm}$ and
    $|P'(x)| > 1/(x-1)^2$ on $I_1$ and $I_2$ {\em property RD}.

If for a word $v$ its associated polynomial $P$ has
    property RD, then the map
    $f_v = f:\{p, z, m\}^\BbN \to \BbR^2$ given by
    $f(w) = \bigl(\beta_1^{(w)}, \beta_2^{(w)}\bigr)$ is well defined.
It is easy to see that such a map is continuous.
It is also easy to see that for those infinite words $w$ which only contain a finite number
    of non-zero terms, the image corresponds to points that are roots of a
    $\{p, z, m\}$ polynomial, and hence such $w$ are not in $\S$.

To see that any such $w$ satisfies the conditions of Lemma~\ref{lem:OP}, let
    $v$ correspond to the coefficients of $P$.
Suppose $w \in \{p, z, m \}^*$.
We see that
    $f_v(w0) = f_v(wvw) = f_v(w\tilde v \tilde w)$.
Thus, if we have a polynomial $P_v$ which satisfies property RD, then
    we can associate with $P_v$ a set of values which are not in $\S$, and whose closure
    is path connected.
We will denote this path connected set by $R_v$.
By Proposition~\ref{prop:S-open}, the complement of $\S$ is closed.
Consequently, $R_v \cap \S = \emptyset$ for all $v$ satisfying
    property RD.

It is easy to see that if $w$ satisfies property RD and $w$ is a prefix of $w'$, then
    $w'$ satisfies property RD as well.
Furthermore, if $w$ is a prefix of $w'$, then $R_{w'} \subset R_{w}$.

\begin{lemma}
Let $w$ satisfy property RD.
Then $f(wm^\infty), f(wp^\infty) \in R_w$.
Furthermore, $R_w$ is contained within the box with sides
    parallel to the axes, and with corners at $f(wm^\infty)$ and $f(wp^\infty)$.
\end{lemma}

We call such a box a {\em bounding box} for $R_w$.
We will also need the concept of a set of bounding boxes for a continuous path.
Let $w w_0$ and $w w_1$ be two points within $R_w$.
By Lemma~\ref{lem:OP}, there is a continuous path from $w w_0$ to $w w_1$ in $R_w$.
Let $k$ be fixed.
To construct this path, we find a series of intermediate points
    $w_{i/2^k}$, each with two addresses.
Each of these addresses is such that
    $w_{i/2^k}$ and $w_{(i+1)/2^k}$ agree on the first $|w| +k$ terms.
Denote these terms by $a_1 a_2 \dots a_k$.

Thus, both these terms are found within the subregions $R_{w a_1 a_2 \dots a_k}$.
Furthermore, by construction, the path from $w_{i/2^k}$ to $w_{(i+1)/2^k}$ will also
    be within this subregion.
Hence this pair, and the path between this pair will be contained within the bounding
    box for $R_{w a_1 a_2 \dots a_k}$.
Taking the union over all of these pairs, we get a series of smaller bounding boxes that
    contain the continuous path from $w w_0$ to $w w_1$.
We will call such a series of boxes the {\em level $k$ bounding boxes} for a path in $R_{w}$.

\begin{lemma}
\label{lem:poly}
The following words satisfy property RD.
\begin{eqnarray*}
w_1 &=& pmmmpzzppzppzppz \\
w_2 &=& pmmmzp^7mz \\
w_3 &=& pmmmzp^7mp \\
w_4 &=& pmmmzp^7zm \\
w_5 &=& pmmmpzzppzpppzpz \\
w_6 &=& pmmmpzzpppmp^4zp.
\end{eqnarray*}
\end{lemma}

\begin{proof}
This is a simple calculation that we leave as an exercise for the reader.
\end{proof}

\begin{lemma}
\label{lem:corner}
The closure of the set of roots generated by the polynomials in
    Lemma~\ref{lem:poly} surrounds $C$.
\end{lemma}

\begin{proof}
To see that $R_{w_1}$ is connected to $R_{w_2}$, consider
    $R_{w_1zpppzpzp}$ and $R_{w_2m^{11}}$.
The former has corners at:
    \[ [1.323453274, 1.648718809], [1.314160784, 1.648757942] \]
    and the latter has corners at:
    \[ [1.321413068, 1.648715950], [1.315100914, 1.648769575]. \]
The path from $ [1.323453274, 1.648718809]$ to $[1.314160784, 1.648757942]$
    must intersect the path from
    $ [1.321413068, 1.648715950]$ to $[1.315100914, 1.648769575]$.
See Figure~\ref{fig:w1w2} for these two sets, and the continuous paths going from
    $f_{w_1zpppzpzp}(p^\infty)$ to $f_{w_1zpppzpzp}(m^\infty)$, and from
    $f_{w_1zpppzpzp}(p^\infty)$ to $f_{w_2m^{11}}(m^\infty)$,
    and the bounding boxes.
\begin{figure}
\includegraphics[width=300pt,height=300pt]{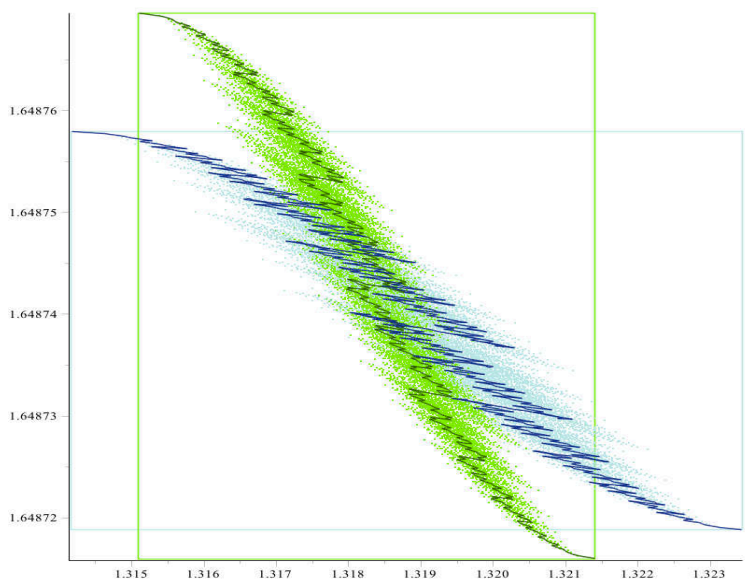}
\caption{$R_{w_1zpppzpzp}$ (green) and $R_{w_2m^{11}}$ (turquoise)}
\label{fig:w1w2}
\end{figure}

To see that $R_{w_2}$ is connected to $R_{w_3}$, we notice that
    \[
    f_{w_2}(pmmmzp^7m) = f_{w_3}(mmmzp^7m).
    \]

To see that $R_{w_3}$ is connected to $R_{w_4}$, we notice that
    \[
    f_{w_3}(ppzm^7) = f_{w_4}(pppzm^7).
    \]

To see that $R_{w_4}$ is connected to $R_{w_5}$, consider
    $R_{w_4m^{14}}$ and $R_{w_5ppzzpppzpmz}$.
The former has corners at:
\[ [1.328228762, 1.646703763], [1.324717957, 1.646712975] \]
    and the latter has corners at:
\[ [1.327323576, 1.646702692], [1.324894555, 1.646715284]. \]
The path from $[1.328228762, 1.646703763]$ to $[1.324717957, 1.646712975]$
    must intersect the path from
    $[1.327323576, 1.646702692]$ to $[1.324894555, 1.646715284]$.
See Figure~\ref{fig:w4w5} and the continuous paths connecting the extreme points of each of these sets.
\begin{figure}
\includegraphics[width=300pt,height=300pt]{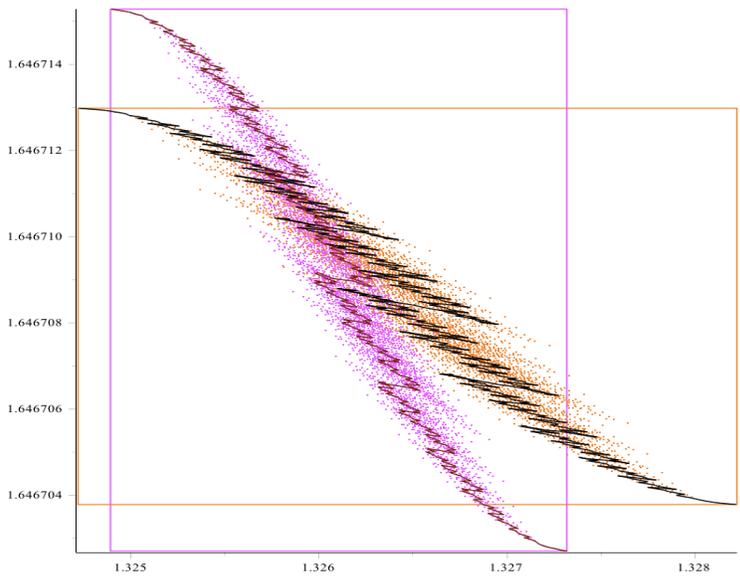}
\caption{$R_{w_4m^{14}}$ (orange) and $R_{w_5ppzzpppzpmz}$ (magenta)}
\label{fig:w4w5}
\end{figure}

For the next two, we need to strengthen the idea of bounding box, as described above.

Consider $R_{w_5mmmp^4mppp}$ and $R_{w_6pz^4zzmzmm}$.
See Figure~\ref{fig:w5w6} and the continuous paths connecting the extreme points of each of these sets as well as
    the level~9 bounding boxes for the path in $R_{w_5mmmp^4mppp}$ and
    the level~2 bounding boxes for the path in $R_{w_6pzm^4zzmzmm}$.
Precise coordinates for the bounding boxes for the continuous paths can be
    found at \cite{HomePage}.
\begin{figure}
\includegraphics[width=300pt,height=300pt]{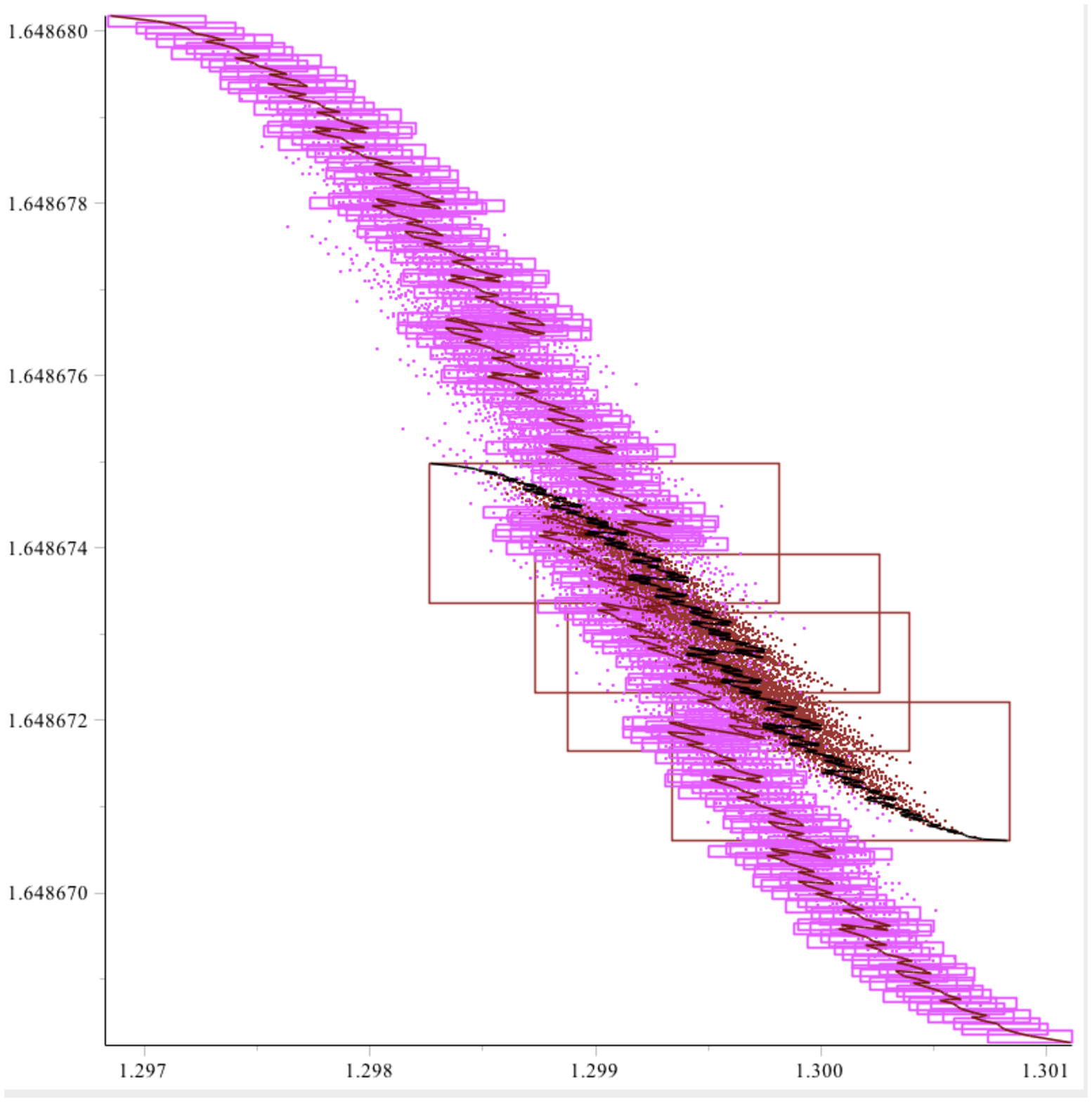}
\caption{$R_{w_5mmmp^4mppp}$ (magenta) and $R_{w_6pzm^4zzmzmm}$ (brown)}
\label{fig:w5w6}
\end{figure}

Finally, consider $R_{w_6mmmp^7}$ and $R_{w_1zppm^4z^5m}$.
See Figure~\ref{fig:w6w1} and the continuous paths connecting the extreme points of each of these sets as well as
    the level~9 bounding boxes for the path in $R_{w_5mmmp^4mppp}$ and
    the level~2 bounding boxes for the path in $R_{w_6pzm^4zzmzmm}$.
Precise coordinates for the bounding boxes for the continuous paths can be
    found at \cite{HomePage}.
\begin{figure}
\includegraphics[width=300pt,height=300pt]{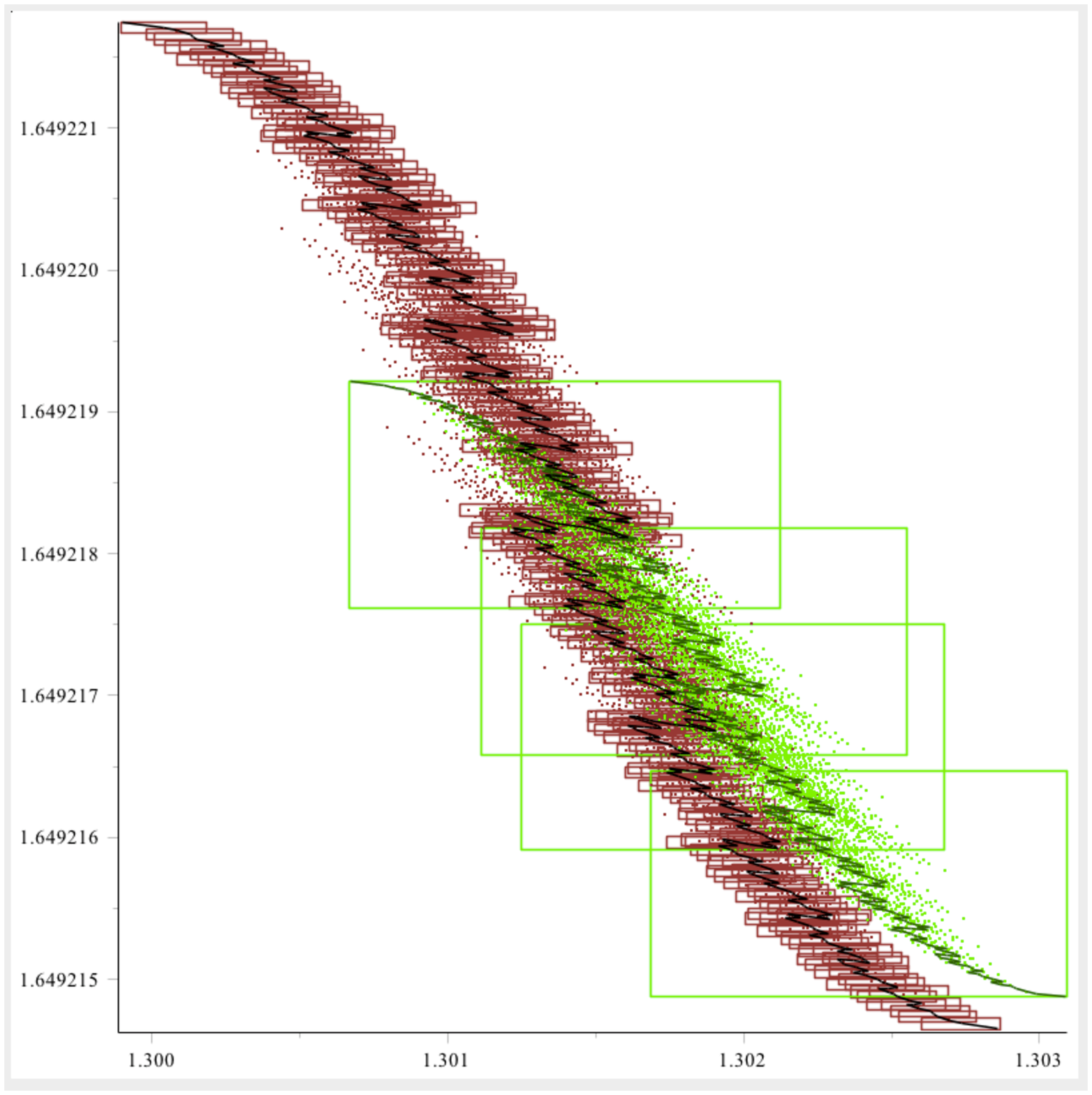}
\caption{$R_{w_6mmmp^8}$ (brown) and $R_{w_1zppm^4z^5m}$ (green)}.
\label{fig:w6w1}
\end{figure}

These surround the region in question, see Figure~\ref{fig:sharks}.
\end{proof}

\begin{rmk}
Visually it appears likely that $R_{w_2}$ intersects $R_{w_4}$ and we probably
    do not need $R_{w_3}$.
\end{rmk}

\begin{cor}
The set $\S$ is not connected.
\end{cor}

\begin{cor}
The connectedness locus $\mathcal N=\S^c$ is not simply connected.
\end{cor}

\begin{rmk}
A method similar to the one described in this section was used in \cite[Section 12]{Bandt02}
to show that a certain connectedness locus is not simply connected (in a different setting).
\end{rmk}

\section{Open questions}
\label{sec:conc}

There is a great deal of questions that this line of research raises, which still remain unanswered.
Here are some of them.

\begin{enumerate}
\item Is it true that if some point of the attractor has a non-empty neighbourhood, then so does (0,0)?
      In particular, what is the precise relationship between $\I$ and $\Z$?

\item We see that if $(0,0) \notin A_{\beta_1, \beta_2}$, then $(\beta_1, \beta_2) \notin \Z$.
      There are examples of $(\beta_1, \beta_2) \notin \Z$ such that that $A_{\be_1,\be_2}$ nonetheless contains $(0,0)$
      -- see Figure~\ref{fig:alg-not-in-z}.
            It would be helpful to find better criteria for a points $(\beta_1, \beta_2) \notin \Z$.

\item Find an example of $\beta_1, \beta_2$ such that
\begin {itemize}
\item $(0,0) \in A_{\beta_1, \beta_2}$;
\item $(0,0) \notin A_{\beta_1, \beta_2}^o$;
\item $(\beta_1, \beta_2) \notin \partial S$.
\end{itemize}

\item Can a point with a unique address be an interior point of $A$?

\item Does the claim in Theorem~\ref{thm:simult}~(ii) hold for all pairs $(\be_1,\be_2)$? Note that
given $\be\in(1,2)$, almost every $x\in (0,1/(\be-1))$ has a continuum of $\be$-expansions \cite{Sid03}, and
furthermore, this continuum can be chosen to have an exponential growth \cite{Kempton}. Thus, one could
hope to adapt our argument so it would hold for $(\be_1,\be_2)$ with both $\be_1$ and $\be_2$ greater than the golden
ratio.

\item We see that $\S \subset \OO$.
      Furthermore, $(\beta_1^{(n)}, \beta_2^{(n)}) \in \partial \S \cap \partial \OO$.
      When approximating $\S$ and $\OO$ computationally, via Lemma \ref{lem:SOcomp}, then
          the level $n$ approximation of $\OO$ is the closure of the level $n$ approximation of $\S$.
      Is $\OO$ the closure of $\S$?

\item Is $\Z \cap \OO = \emptyset$?

\item Justify the `spikes' in $\S$ near $(1,2)$ and $(2,1)$. That is, we know that both corners are limit points of
      $\S$ (Theorem~\ref{thm:S struct}); is it true that for any $h>0$ there exists a point
      $(\beta_1, \beta_2)$ in $(2-h,2) \times (1,1+h)$ which is not in
      $\S$?
      By looking at $(\beta_1^{(n)}, \beta_2^{(n)})$ we get a partial idea of the structure of $\S$ near $(1,2)$,
      but not a complete picture.

\item As mentioned at the beginning of Section~\ref{sec:island}, $(\be_1,\be_2)\in\mathcal S$, where
$\be_1=1.335438104, \be_2=1.646743824$. Thus, we have $\be_1+\be_2=2.982181928$, i.e.,
some small chunk of $\mathcal S$ lies below the diagonal (which is not at all obvious
from Figure~\ref{fig:S40}). It would be interesting to find
the smallest $\e>0$ such that
 $\S \subset \{(\beta_1, \beta_2) : \beta_1 + \beta_2 > 3-\e\}$ -- see Figure~\ref{fig:S-diag}.

 \begin{figure}
\includegraphics[width=300pt,height=300pt]{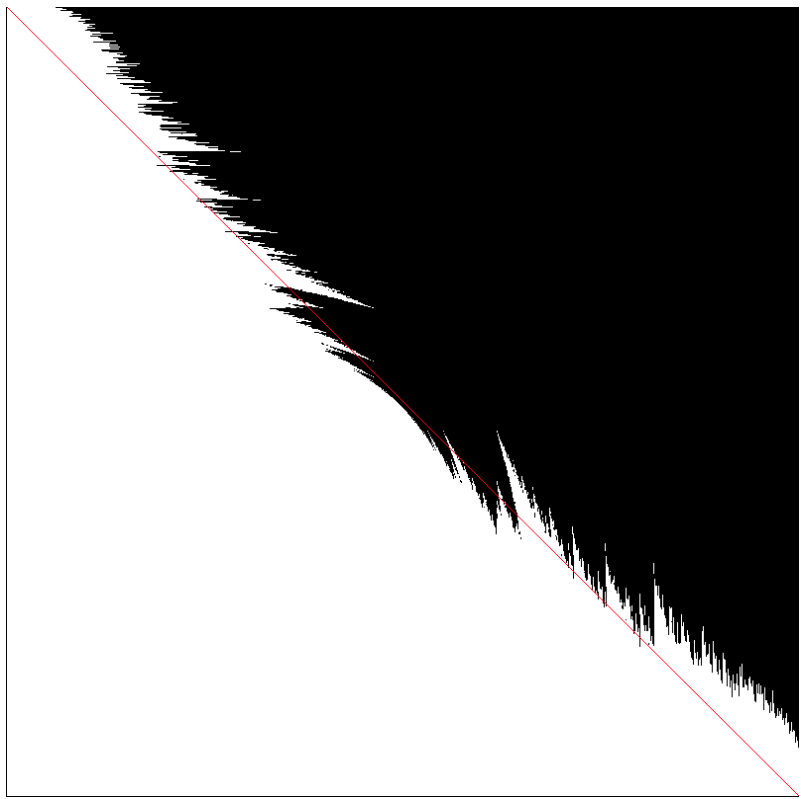}
\caption{The set $\mathcal S$ together with the diagonal $\be_1+\be_2=3$. (Level~20 approximation.)}
\label{fig:S-diag}
\end{figure}

\item We know that $\S$ contains at least three disjoint components (by symmetry around the line
    $\beta_1 = \beta_2$).
    Does it contain a finite number of components, or an infinite number of components?

\item Prove or disprove that for sufficiently small $\beta_1$ and $\beta_2$ the attractor
$A_{\beta_1, \beta_2}$ is simply connected.

\item Show that the lower box (or Hausdorff) dimension of $\partial A_{\beta_1, \beta_2}$ is strictly greater than~1
for all $\beta_1, \beta_2$.

\end{enumerate}

\section*{Acknowledgements}

The authors would like to thank Boris Solomyak and the anonymous referee
    for many helpful comments and suggestions.

\end{document}